\documentclass[11pt]{amsart}
\usepackage{amssymb, latexsym, mathrsfs,   color }
\usepackage[all]{xy}
\usepackage[colorlinks=true, pdfstartview=FitV,linkcolor=blue,citecolor=blue,urlcolor=blue]{hyperref}
\usepackage [latin1]{inputenc}
\usepackage{chngcntr}
\counterwithin{figure}{section}
    \advance \topmargin by -\headheight
    \advance \topmargin by -\headsep

    \evensidemargin \oddsidemargin
\newtheorem {theorem}    {Theorem}[section]

\newtheorem {lemma}      [theorem]    {Lemma}

\newtheorem {proposition}[theorem]    {Proposition}

\newcommand{\bb}{\mathbb}
\renewcommand{\rm}{\mathrm}

\newcommand{\cal}{\mathcal}

\newcommand{\GGL}{\mathrm{GL}}
\newcommand{\UU}{\mathrm{U}}
\newcommand{\Fq}{\bb{F}_q}
\newcommand{\fq}{(\bb{F}_q)}
\theoremstyle{definition}

\newcommand{\so}{\mathrm{SO}}

\renewcommand{\sp}{\mathrm{Sp}}

\newcommand{\BB}{\mathbb}

\numberwithin{equation}{section}

\begin{document}

\title{Lusztig correspondence and the finite Gan-Gross-Prasad problem}

\date{\today}

\author[Zhicheng Wang]{Zhicheng Wang}

\address{School of Mathematical Science, Suzhou University, Suzhou 310027, Jiangsu, P.R. China}

\email{11735009@zju.edu.cn}
\subjclass[2010]{Primary 20C33; Secondary 22E50}

\begin{abstract}
The Gan-Gross-Prasad problem is to describe the restriction of representations of a classical group $G$ to smaller groups $H$ of the same kind. In this paper, we solved the Gan-Gross-Prasad problem over finite fields completely. In previous work \cite{LW1,LW2,LW3,Wang}, we study the Gan-Gross-Prasad problem for unipotent representations of finite classical groups. The main tools used are the Lusztig correspondence as well as a formula of Reeder \cite{R} for the pairings of Deligne-Lusztig characters. We give a reduction decomposition of Reeder's formula, and deduce the Gan-Gross-Prasad problem for arbitrary representations from the unipotent representations by Lusztig correspondence.
\end{abstract}

\maketitle

\section{Introduction}\label{sec1}

In \cite{GP1}, Gross and Prasad formulated conjectures which relate the periods of automorphic forms on $\so_{n+1}\times \so_n$ along the diagonal subgroup $\so_n$ with central $L$-values, where $\so_{n+1}$ and $\so_n$ are special orthogonal groups over a local field. In \cite{GGP1}, Gan, Gross, and Prasad extended these local conjectures to all classical groups. The local Gan-Gross-Prasad conjecture has been resolved for orthogonal groups by J.-L. Waldspurger and C. M\oe glin \cite{W2, W3, W4, MW}, for unitary groups by R. Beuzart-Plessis \cite{BP1, BP2} and W. T. Gan and A. Ichino \cite{GI}, and for symplectic-metaplectic groups by H. Atobe \cite{Ato}.

The main goal of this paper is to consider the Gan-Gross-Prasad conjecture over finite fields. Partial results can be proved via known results for local fields \cite[section 4,5]{GGP2}. However, that way does not allow one to get the complete answer to the finite Gan-Gross-Prasad problem. In previous works \cite{LW1,LW3}, we studied the Gan-Gross-Prasad problem of unipotent representations of finite unitary groups and in \cite{LW2,Wang} for finite orthogonal groups and finite symplectic groups. This paper will give a formula reducing the Gan-Gross-Prasad problem of arbitrary representations to unipotent representations.

There are two types of periods in the Gan-Gross-Prasad conjecture: the Bessel periods and the Fourier-Jacobi periods. One can deduce the Fourier-Jacobi case from the Bessel case by the standard arguments of theta correspondence and see-saw dual pairs, which are used in the proof of the local Gan-Gross-Prasad conjecture (see \cite{GI, Ato}). Theta correspondence over finite fields is described explicitly in \cite{AMR, P1, P2}. So in this paper, we only focus on the Bessel case.

\subsection{Representation of finite reductive groups}

Let $\overline{\mathbb{F}}_q$ be an algebraic closure of a finite field $\mathbb{F}_q$, which is of characteristic $p>2$. Let $G$ be a connected reductive algebraic group defined over $\Fq$, with Frobenius map $F$, and let $Z$ be the center of $G^F$. We assume that $q$ is large enough such that the main theorem in \cite{S2} holds, namely we assume that for every $F$-stable maximal torus $T$ of $G$, $T^F/Z$ has at least two Weyl group orbits of regular characters.

 Let $H$ be a subgroup of $G$, and $\pi\in\rm{Irr}(G^F)$ and $\sigma\in\rm{Irr}(H^F)$. We write
\[
\langle \pi,\sigma \rangle_{H^F} = \dim \mathrm{Hom}_{H^F}(\pi,\sigma ).
\]
Let $P$ be an F-stable parabolic subgroup of $G$ with Levi factor $L$. Let $\sigma\in\rm{Irr}(L^F)$. We write the parabolic induction
\[
I^G_P(\sigma):=\rm{Ind}^{G^F}_{P^F}(\sigma).
\]

For an $F$-stable maximal torus $T$ of $G$ and a character $\theta$ of $T^F$,  let $R_{T,\theta}^G$ be the virtual character of $G^F$ defined by Deligne and Lusztig in \cite{DL}. We say a complex irreducible representation is uniform if its character is a linear combination of the Deligne-Lusztig characters. Let $G^*$ be the dual group of $G$. For simplicity, we still denote the Frobenius endomorphism of $G^*$ by $F$. Then there is a natural bijection between the set of $G^F$-conjugacy classes of $(T, \theta)$ and the set of $G^{*F}$-conjugacy classes of $(T^*, s)$ where $T^*$ is a $F$-stable maximal torus in $G^*$ and $s \in   T^{*F}$. If $(T, \theta)$ corresponds to $(T^*, s)$, we denote $R_{T,\theta}^G$ by $R_{T^*,s}^G$.

From now on, we will consider the following three cases: general linear groups, unitary groups, and special orthogonal groups. Let $a\in \overline{\mathbb{F}}_q^*$ and
  \[
  [ a] :=\left\{
  \begin{aligned}
 &\{a^{{(-q)}^k}|k\in\bb{Z}\}&\textrm{ if }G\textrm{ is a unitary group};\\
 &  \{a^{{q}^k}|k\in\bb{Z}\}&\textrm{ otherwise}.
\end{aligned}
\right.
  \]
  For a pair $(T^*,s)$ corresponding to the pair $(T,\theta)$, assume that $s$ has eigenvalues $x_1,\cdots,x_n$. If $[a]\nsubseteq\{x_1,\cdots,x_n\}$, then we say $[a]\notin s$ or $[a]\notin \theta$. If there exists an $[a]$ such that $x_i\in[a]$ for each $i$, then we say $s\sim[a]$ or $\theta\sim[a]$. For abbreviation, we write $\pm 1\notin s$, $s\sim\pm1$ and $\theta\sim\pm1$ instead of $[\pm 1]\notin s$, $s\sim[\pm1]$ and $\theta\sim[\pm1]$, respectively. Here $-1$ is the unique element in $\Fq$ such that $(-1)^2=1$ and $-1\ne 1$.

For a semisimple element $s \in G^{*F}$, define the Lusztig series as follows:
\[
\mathcal{E}(G^F,s) = \{ \pi \in \rm{Irr}(G^F)  :  \langle \pi, R_{T^*,s}^G\rangle \ne 0\textrm{ for some }T^*\textrm{ containing }s \}.
\]
And it is known that $\rm{Irr}(G^F)$ is partitioned into Lusztig series
\[
\rm{Irr}(G^F)=\coprod_{(s)}\mathcal{E}(G^F,s),
\]
where $(s)$ runs over the conjugacy classes of semisimple elements. Moreover, there is a bijection
\[
\mathcal{L}_s:\mathcal{E}(G^F,s)\to \mathcal{E}(C_{G^{*F}}(s),1).
\]
Note that usually, the correspondence $\mathcal{L}_s$ is not uniquely determined. For unitary groups and general linear groups, every representation is uniform. For special orthogonal groups, in \cite[Theorem 1.4]{Wang}, we reduce the Gan-Gross-Prasad problem of arbitrary representations to the representation $\pi\in \mathcal{E}(G^F,s)$ with $\pm1\notin s$. Note that in the special orthogonal groups case, $\pm1\notin s$ implies that $\pi$ is uniform. So we only consider the Lusztig correspondence for the uniform case, and $\mathcal{L}_s$ is uniquely determined by (\ref{l}).

 It is well known that the centralizer $C_{G^*(\overline{\bb{F}}_q)}(s)$ is a product of classical groups. We adopt the notations of \cite[subsection 1.B]{AMR}. Let $T^*(\overline{\bb{F}}_q) \cong \overline{\bb{F}}^\times_q\times\cdots\times\overline{\bb{F}}_q^\times$ be a rational maximal torus of $G^*(\overline{\bb{F}}_q)$, and $s=(x_1,\cdots,x_l)\in T^{*F}$. For $a\in \overline{\mathbb{F}}_q^*$, let $\nu_{a}(s):=\#\{i|x_i=a\}$. The group $C_{G^*(\overline{\bb{F}}_q)}(s)$ has a natural decomposition:
\[
C_{G^*(\overline{\bb{F}}_q)}(s)=\prod_{[ a] \subset \{x_i\}}G^*_{[ a]}(s)(\overline{\bb{F}}_q)
\]
where $G^*_{[ a]}(s)(\overline{\bb{F}}_q)$ is a reductive quasi-simple group of rank equal to $\#[ a]\cdot\nu_{a}(s)$. Now we have the following situations:
\begin{itemize}

\item If $a=\pm1$ and $G=\GGL_n$, $\UU_n$, $\so^\epsilon_{2n}$ or $\sp_{2n}$, then $G_{[ a]}(s)(\overline{\bb{F}}_q)$ is $\GGL_m$, $\UU_m$, $\so^\epsilon_{2m}$ or $\sp_{2m}$, respectively, for some $m\le n$;

\item If $a\ne \pm 1$ and $G=\GGL_n$, then $G_{[ a]}(s)(\overline{\bb{F}}_q)$ is a general linear group;

\item  If $a\ne \pm 1$ and $G\ne\GGL_n$, then $G_{[ a]}(s)(\overline{\bb{F}}_q)$ is either a general linear group or a unitary group.

\end{itemize}

For $[a]\nsubseteq \{x_i\}$, we set $G_{[ a]}(s):=1$. Then we have
\[
C_{G^{*F}}(s)=\prod_{[ a] }G^{*F}_{[ a]}(s).
\]
Hence, we can rewrite the Lusztig correspondence as follows:
\begin{equation}\label{luspi}
\cal{L}_s(\pi)=\prod_{[a]}\pi_{[a]}
\end{equation}
where $\pi_{[a]}$ is a unipotent representation of $G^{*F}_{[ a]}(s)$ for each $[a]$.

\subsection{The main result}

Let $V_n$ be an $n$-dimensional space over $\mathbb{F}_{q}$ with a nondegenerate symmetric (resp. Hermitian) bilinear form $(,)$, which defines the special orthogonal group $\so(V_n)$ (resp. unitary group $\UU_n$), and let $V_{n-m}\subset V_n$. Assume that $n-m$ is odd if $(,)$ is symmetric. Let $\pi$ and $\sigma$ be complex irreducible representations of $\so(V_n)$ and $\so(V_{n-m})$ (resp. $\UU(V_n)$ and $\UU(V_{n-m})$), respectively.
The Gan-Gross-Prasad problem is concerned with the multiplicity
\begin{equation}\label{ggp}
m(\pi, \sigma):=\langle \pi\otimes\bar{\nu},\sigma \rangle_{H^F} = \dim \mathrm{Hom}_{H^F}(\pi\otimes\bar{\nu},\sigma )
\end{equation}
where the data $(H, \nu)$ is defined as (1.2) in \cite{LW1} and (1.2) in \cite{LW2}.
According to the parity of $n-m$, the above Hom space is called the Bessel model or the Fourier-Jacobi model. We call the above two models the basic cases if $n-m\le1$.

Now we focus on the Bessel case. For unitary groups, $\pi$ and $\sigma$ are always uniform. For special orthogonal groups, in \cite[Theorem 1.4]{Wang}, we reduce the finite Gan-Gross-Prasad problem of arbitrary representations to the case $\pi$ and $\sigma$ in the Lusztig series associated with the semisimple elements $s$ and $s'$ such that $\pm1\notin s$ and $\pm1\notin s'$, respectively. So the finite Gan-Gross-Prasad problem is solved if one can solve the above case. Note that in the special orthogonal groups case, $\pm1\notin s$ and $\pm1\notin s'$ imply $\pi$ and $\sigma$ are uniform.
We can calculate the multiplicity by Reeder's formula (\cite{R, LW1}). This formula in \cite{R} is used to calculate $\langle R^G_{T,\chi},R^H_{S,\eta}\rangle_{H^F}$ for some groups $G$ and $H$ with regular characters $\chi$ and $\eta$, and verify some conjectures in \cite{GP1} for depth-zero supercuspidal representations of special orthogonal groups over a local field. In general, explicit calculation with Reeder's formula is still quite involved. The first main theorem of this paper is the following reduction decomposition of Reeder's formula:
\begin{theorem}\label{main1}
(i) Let $G=\UU_{n+1}$ and $H=\UU_n$. Let $T$ and $S$ be $F$-stable maximal torus of $G$ and $H$, respectively, and $\chi$ and $\eta$ be characters of $T^F$ and $S^F$, respectively. Suppose that $(T,\chi)$ and $(S,\eta)$ correspond to $(T^*,t)$ and $(S^*,s)$, respectively. Then
\begin{equation}\label{main11}
\epsilon_{t,s}\langle R^{G}_{T,\chi},R_{S,\eta}^{H}\rangle _{H^F}=\prod_{[a]}\epsilon_{[a]}\langle R^{G'[a]}_{T'[a],\theta[a]\otimes1},R^{H'[a]}_{S'[a],1}\rangle _{H^{\prime }[a]^F}
\end{equation}
where $\epsilon_{t,s}$ and $\epsilon_{[a]}\in\{\pm1\}$ are defined in (\ref{ua}) and (\ref{uts}), respectively, and $G'[a]$, $T'[a]$, $\theta[a]$, $H'[a]$, and $S'[a]$ are defined in the subsection \ref{sec6.2}. Moreover, $G'[a]$ and $H'[a]$ are either general linear groups or unitary groups.

(ii) Let $V$ be an $2n+1$ dimensional space over $\Fq$ with a nondegenerate symmetric bilinear form $(,)$. Let $v\in V$ with $(v,v)\ne 0$, and let $U$ be the orthogonal space of $v$ in $V$. We take $G=\so(V)$ and $H=\so(U)$. Then $G\cong\so_{2n+1}$ and $H\cong\so_{2n}^\epsilon$. Let $T$ and $S$ be $F$-stable maximal torus of $G$ and $H$, respectively and let $\chi$ and $\eta$ be characters of $T^F$ and $S^F$, respectively. Suppose that $(T,\chi)$ and $(S,\eta)$ correspond to $(T^*,t)$ and $(S^*,s)$, respectively, and $\pm1\notin t$ and $\pm1\notin s$. Then
\begin{equation}\label{main12}
\epsilon_{t,s}\langle R^{G}_{T,\chi},R_{S,\eta}^{H}\rangle _{H^F}=\prod_{[a]}\epsilon_{[a]}\langle R^{G'[a]}_{T'[a],\theta[a]\otimes1},R^{H'[a]}_{S'[a],1}\rangle _{H^{\prime }[a]^F}
\end{equation}
where $\epsilon_{t,s}$ and $\epsilon_{[a]}\in\{\pm1\}$ are defined in (\ref{soa}) and (\ref{sots}), respectively, and $G'[a]$, $T'[a]$, $\theta[a]$, $H'[a]$, and $S'[a]$ are defined in subsection \ref{sec7.1}. Moreover, $G'[a]$ and $H'[a]$ are either general linear groups or unitary groups.
\end{theorem}

The RHS of (\ref{main11}) and (\ref{main12}) have been calculated in \cite{LW1} and used to solve the finite Gan-Gross-Prasad problem for unipotent representations of unitary groups. In other words, by Theorem \ref{main1}, we can reduce the the multiplicity (\ref{ggp}) to a product of some known multiplicities of unipotent representations for finite unitary groups and finite general linear groups.

Since some multiplicities of certain irreducible representations of general linear groups appear in our main results, we shall explain them. We emphasize that these multiplicities for general linear groups are \emph{not} the multiplicity in the Gan-Gross-Prasad problem. Let $\pi$ and $\sigma$ be two irreducible unipotent representations of $\GGL_n\fq$ and $\GGL_{m}\fq$ with $n\ge m$, respectively. Here we do not require $n-m$ odd. We set
\begin{equation}\label{defgl}
m(\pi, \sigma):=\langle \pi,I^{\GGL_{n+1}}_{P}(\tau\otimes\sigma) \rangle_{\GGL_n(\Fq)}
\end{equation}
where $P$ is an $F$-stable maximal parabolic subgroup with levi factor $\GGL_{n+1-m}\times\GGL_m$ and $\tau$ is a generic representation of $\GGL_{n+1-m}\fq$ (see subsection \ref{sec6.3} for details).

Let $G_n=\UU_n$ or $\GGL_n$. Assume that $n\ge m$. Let $\pi$ and $\sigma$ are irreducible unipotent representations of $G_n^F$ and $G_m^F$, respectively. We know that $m(\pi, \sigma)=m( \sigma,\pi)$ with $n=m$ for unitary groups in \cite{LW3}. According to Lemma \ref{gl}, for general linear groups, this is also true. So we can rewrite the above multiplicity and the multiplicity in the Gan-Gross-Prasad problem as follows
  \[
m(\pi, \sigma):=
\left\{
  \begin{aligned}
 &m(\pi, \sigma)&\textrm{ if }n\ge m;\\
  &m(\sigma,\pi)&\textrm{ if }n< m.
\end{aligned}
\right.
 \]

By Theorem \ref{main1}, we can calculate the multiplicity in the basic cases of the Bessel case of the
Gan-Gross-Prasad problem.
\begin{theorem}\label{main2}
Let $G$ and $H$ be the classical groups in Theorem \ref{main1}. Let $\pi$ and $\sigma$ be irreducible representations of $G^F$ and $H^F$, respectively. Suppose that $\pi\in\cal{E}(G^F,t)$ and $\sigma\in\cal{E}(H^F,s)$. If $G$ and $H$ are special orthogonal groups, we additionally assume that $\pm1\notin t$ and $\pm1\notin s$ (In this case, $\pi$ and $\sigma$ are uniform representations, and $G'[a]$ and $H'[a]$ are all general linear groups and unitary groups).
Then
\[
\left\langle\pi,\sigma\right\rangle_{H^F}  =\prod_{[a]}
 m(\pi[a],\sigma[a])
 \]
 where $\pi[a]$ and $\sigma[a]$ are defined in (\ref{luspi}), and they are both unipotent representations of general linear groups or unitary groups.
\end{theorem}

Suppose that $\pi\in\cal{E}(G^F,t)$ and $\sigma\in\cal{E}(H^F,s)$. For $n-m>1$, we can make a reduction to the basic case as follows. Let $l=\frac{n+1-m}{2}$, and let $P$ be an $F$-stable maximal parabolic subgroup of $\rm{SO}(V_{n+1})$ (resp. $\UU(V_{n+1})$) with Levi factor $\GGL_l\times \rm{SO}(V_m)$ (resp. $\left(\rm{Res}_{\bb{F}_{q^2}/\Fq} \GGL_l\right)\times\UU(V_m)$). By \cite[Proposition 5.2]{LW1} and \cite[Proposition 5.2]{LW2}, there exists a cuspidal representation $\tau$ of $\GGL_l\fq$ satisfying some technical conditions such that the following equation holds
\begin{equation}\label{2o}
m(\pi,\sigma)=\langle\pi,I^{\rm{SO}(V_{n+1})}_{P}(\tau\otimes\sigma)\rangle_{\rm{SO}(V_{n})} \ \left(\textrm{resp. $m(\pi,\sigma)=\langle\pi,I^{\rm{U}(V_{n+1})}_{P}(\tau\otimes\sigma)\rangle_{\rm{U}(V_{n})}$}\right)
\end{equation}
where $I^{\rm{SO}(V_n+1)}_{P}(\tau\otimes\sigma)$ (resp. $I^{\rm{U}(V_{n+1})}_{P}(\tau\otimes\sigma)$) is the parabolic induction. Moreover, by our assumption that $q$ is large enough, we can pick a $\tau$ such that the parabolic induction $I^{\rm{SO}(V_{n+1})}_{P}(\tau\otimes\sigma)$ (resp. $I^{\rm{U}(V_{n+1})}_{P}(\tau\otimes\sigma)$) is an irreducible representation and $I^{\rm{SO}(V_{n+1})}_{P}(\tau\otimes\sigma)\in \cal{E}(\rm{SO}(V_{n+1}),(s',s))$ (resp. $I^{\rm{U}(V_{n+1})}_{P}(\tau\otimes\sigma)\in \cal{E}(\rm{U}(V_{n}),(s',s))$ where $s'$ has no common eigenvalues with $s$ and $t$.
So the RHS of (\ref{2o}) can be calculated by Theorem \ref{main2}.

\begin{theorem}\label{main3}
Let $\pi$ be an irreducible representation of $\so(V_n)$ (resp. $\UU(V_n)$) and $\sigma$ be an irreducible representation of $\so(V_m)$ (resp. $\UU(V_m)$). Suppose that $n-m$ is odd. If $G$ and $H$ are special orthogonal groups, we additionally assume that $\pm1$ is not a eigenvalue of $t$ and $s$.
Then
\[
m(\pi,\sigma)  =\prod_{[a]}
 m(\pi[a],\sigma[a])
 \]
 where $\pi[a]$ and $\sigma[a]$ are defined in (\ref{luspi}), and they are both unipotent representations of general linear groups or unitary groups.
\end{theorem}

 This paper is organized as follows. Section \ref{sec2} contains some preliminaries. In Section \ref{sec3}, we recall the description of $F$-stable maximal tori. Section \ref{sec4} recalls the Lusztig correspondence, and we decompose them through the decomposition of the centralizer $C_{G^*(\overline{\bb{F}}_q)}(s)$. In Section \ref{sec5}, we recall Reeder's formula in \cite{R,LW1}.We prove Theorem \ref{main1} for unitary groups in Section \ref{sec6} and for special orthogonal groups in Section \ref{sec7}. In Section \ref{sec8}, we prove Theorem \ref{main2}.

{\bf Acknowledgement.} The author would like to thank Dongwen Liu for many helpful and enlightening discussions on this topic.

\section{Preliminaries}\label{sec2}
\subsection{Partitions}
Let $\mu=(\mu_1,\mu_2,\cdots,\mu_k)$ be a partition of $n$, and $|\mu|:=\sum_i\mu_i$. For simplicity of notations, we write $\mu=(a_1^{k_1},\cdots,a_l^{k_l})$ where $a_j$ runs over $\{\mu_i\}$ and $k_j=\#\{i|\mu_i=a_j\}$.

If there exists an $h\in\bb{N}$ such that $h|\mu_i$ for each $i$, then we say $h|\mu$ and set $\frac{\mu}{h}:=(\frac{\mu_1}{h},\cdots,\frac{\mu_k}{h})$. If $b\in\{\mu_i\}$, then we say $b\in \mu$. Let $\lambda=(\lambda_1,\cdots\lambda_n)$ be a partition of $m$ with $m\le n$. We say $\lambda\subset\mu$ if $\{\lambda_i\}\subset\{\mu_i\}$. Here $\{\lambda_i\}\subset\{\mu_i\}$ means the containment of multisets of integers.
Let $(\lambda,\mu)$ be a pair of partitions. We say $(\mu',\lambda')\subset(\mu,\lambda)$ if $\mu'\subset\mu$ and $\lambda'\subset\lambda$.

For two partitions $\mu'\subset\mu$, we write
 \begin{equation}\label{pair}
 \mu=[n_1^{a_1}, \ldots, n_l^{a_l}]\quad \rm{and}\quad \mu'=[n_1^{b_1},\ldots, n_l^{b_l}],
 \end{equation}
where $n_i$'s are distinct so that $0\leq b_i\leq a_i$, $i=1,\ldots, l$. We set
\begin{equation} \label{CC}
C_{\mu,\mu'}=\prod_{i=1}^l
{a_i \choose b_i}.
\end{equation}

For a finite set of partitions $\{\mu^i=(\mu^i_1,\cdots,\mu^i_{k_1})\}$ with $i=1,\cdots l$, let $\mu$ be a partition such that $\{\mu_i\}=\bigcup_{j=1}^l\{\mu_i^j,\cdots,\mu^j_{k_j}\}$. Here $\{\mu_i\}=\bigcup_{j=1}^l\{\mu_i^j,\cdots,\mu^j_{k_j}\}$ means the equivalence of multisets of integers. We say $\mu=\bigcup_i\mu^i$ or $\mu=\mu^1\bigcup\mu^2\bigcup\cdots\bigcup\mu^l$.

\subsection{Centralizer of a semisimple element }\label{sec2.2}

 Let $G$ be a classical group defined over a finite field $\Fq$ and $F$ be a Frobenius morphism of $G$. Let $s$ be a semisimple element in the connected component of $G$, and $C_{G(\overline{\bb{F}}_q)}(s)$ be the centralizer in $G(\overline{\bb{F}}_q)$ of a semisimple element $s \in G^0(\overline{\bb{F}}_q)$. In \cite[subsection 1.B]{AMR}, A.-M. Aubert, J. Michel and R. Rouquier described $C_{G(\overline{\bb{F}}_q)}(s)$ as follows. Let $T(\overline{\bb{F}}_q) \cong \overline{\bb{F}}^\times_q\times\cdots\times\overline{\bb{F}}_q^\times$ be a rational maximal torus of $G(\overline{\bb{F}}_q)$, and $s=(x_1,\cdots,x_l)\in T^F$. For $a\in \overline{\bb{F}}_q^*$, let $\nu_{a}(s):=\#\{i|x_i=a\}$ and
  \[
  [ a] :=\left\{
  \begin{aligned}
 &\{a^{{(-q)}^k}|k\in\bb{Z}\}&\textrm{ if }G\textrm{ is a unitary group};\\
 &  \{a^{{q}^k}|k\in\bb{Z}\}&\textrm{ otherwise}.
\end{aligned}
\right.
  \]
  Clearly, if $a'\in[a]$ and $a\in\{x_i\}$, then $a'\in\{x_i\}$ and $\nu_{a'}(s)=\nu_{a}(s)$.
 The group $C_{G(\overline{\bb{F}}_q)}(s)$ has a natural decomposition with the eigenvalues of $s$:
\[
C_{G(\overline{\bb{F}}_q)}(s)=\prod_{[ a] \subset \{x_i\}}G_{[ a]}(s)(\overline{\bb{F}}_q)
\]
where $G_{[ a]}(s)(\overline{\bb{F}}_q)$ is a reductive quasi-simple group of rank equal to $\#[ a]\cdot\nu_{a}(s)$. Now we have the
following situations:
\begin{itemize}

\item If $a=\pm1$ and $G=\GGL_n$, $\UU_n$, $\so^\epsilon_{2n}$ or $\sp_{2n}$, then $G_{[ a]}(s)(\overline{\bb{F}}_q)$ is $\GGL_m$, $\UU_m$, $\so^\epsilon_{2m}$ or $\sp_{2m}$, respectively, for some $m\le n$.

\item If $a\ne \pm 1$ and $G=\GGL_n$, then $G_{[ a]}(s)(\overline{\bb{F}}_q)$ is a general linear group.

\item  If $a\ne \pm 1$ and $G\ne\GGL_n$, then $G_{[ a]}(s)(\overline{\bb{F}}_q)$ is a general linear group or unitary group.

\end{itemize}

For $[a]\nsubseteq \{x_i\}$, we set $G_{[ a]}(s):=1$. Then we rewrite
\[
C_{G^F}(s)=\prod_{[ a] }G^F_{[ a]}(s).
\]

\section{Remarks on maximal tori} \label{sec3}
Let $G$ be a connected reductive algebraic group over $\Fq$, and $F$ be a Frobenius morphism of $G$. Fix an $F$-stable maximal torus $T_0$ in $G$ which is contained in an $F$-stable Borel subgroup of $G$, and let $W=W_G(T_0)=N_G(T)/T$ be the Weyl group of $G$.

For any $F$-stable maximal torus $T$, there is $g\in G$ such that $^gT=T_0$. Since $T$ is $F$-stable, we have $gF(g^{-1})\in N_G(T_0)$. If $w$ is the image of  $gF(g^{-1})$ in $W$, then we denote $T$ by $T_w$. Hence there is a bijection between the set of $G^F$-conjugacy classes of $F$-stable maximal tori of $G$ and the set of $F$-conjugacy classes of $W$ (\cite[Proposition 3.3.3]{C}). It is known that $T_0^{wF}:= \{t \in T_0 : F (t) = {}^{w^{-1}} t\}$ is isomorphic to $T_w^F$ (\cite[the proof of Proposition 3.3.6]{C}), and $W_{G}(T_w)^F\cong C_{W,F}(w)$ where $C_{W,F}(w):=\{w'\in W|w^{\prime -1}wF(w')=w\}$.

\subsection{The general linear groups}\label{sec3.1}
We first consider $G = \GGL_n$. In this case, $W \cong  \BB{S}_n$, where the symmetric group $\BB{S}_n$ acts on the set $\{1,\cdots, n\}$. So the conjugacy classes of $W$ are parameterized by partitions $\mu = (\mu_1, \cdots , \mu_k)$ of $n$. Since the action of the Frobenius morphism $F$ on the Weyl group $W$ is trivial, the $F$-conjugacy classes of $W$ are also parameterized by partitions. We choose a canonical representative $w_{\mu}$ of the conjugacy class corresponding to $\mu$ such that
\[
w_\mu= (1,2,\cdots,\mu_1)(\mu_1+1,\cdots,\mu_1+\mu_2)\cdots(\mu_1+\cdots+\mu_{k-1}+1,\cdots,\mu_1+\cdots+\mu_{k}).
\]

The torus $T_0$ of $\GGL_n$ can be chosen to be the group of diagonal matrices. So for each $t_0 \in T_0$, we write $t_0 = (x_1, \cdots , x_n)$, where $x_1,\cdots, x_n \in \overline{\bb{F}}_q^{\times}$. The Weyl group $W$ acts on $T_0$ by permuting $(x_1,\cdots, x_n)$ in a similar way, that is,
\[
w_\mu\cdot t_0 = (x_{\mu_1},x_1,x_2,x_3,\cdots, x_{\mu_1-1},  x_{\mu_1+\mu_2}, x_{\mu_1+1},x_{\mu_1+2},x_{\mu_1+3},\cdots, x_{\mu_1+\mu_2-1}, x_{\mu_1+\mu_2+\mu_3}\cdots).
\]
For a partition $\mu=(\mu_1,\cdots,\mu_k)=(n_1^{k_1},n_2^{k_2},\cdots,n_r^{k_r})$, we write $w_\mu$ as a product of cycles:
\[
w_\mu=\prod_{i=1}^{r}\prod_{\ell=1}^{k_i}c_{n_i,\ell}.
\]
Let $\psi:\BB{S}_{k_1}\times\cdots\times \BB{S}_{k_r} \to \BB{S}_n$ such that for any $\tau=\tau_{1}\times\cdots\times\tau_{r}\in \BB{S}_{k_1}\times\cdots\times \BB{S}_{k_r} $, the permutation $\psi(\tau)$ sends $c_{n_i,\ell}$ to $ c_{n_i,\tau_i(\ell)}$ for each index $1\le i \le r$ and $1\le\ell\le k_i$. It is well known that
 \begin{equation}\label{cw}
 C_{W,F}(w_\mu)=\left\langle\psi(\BB{S}_{k_1}\times\cdots\times \BB{S}_{k_r}),\prod_{i=1}^{r}\prod_{\ell=1}^{k_i}\langle c_{i,\ell}\rangle\right\rangle
 \end{equation}
 and
 \[
 \# C_{W,F}(w_\mu)=\prod_{i=1}^rn_i^{k_i}k_i!.
 \]

For abbreviation, we write $T_{\mu}$ instead of $T_{w_\mu}$. Let $t\in T_\mu^F$ be a semisimple element. There is a natural embedding:
  \[
\begin{matrix}
T_\mu= \prod_{[a]}T_{\mu[t,[a]]}\hookrightarrow  C_{G^F}(t)= \prod_{[a]}G_{[a]}(t).

\end{matrix}
\]
where $T_{\mu[t,[a]]}=T_\mu\bigcap G_{[a]}(t)$ is an $F$-stable maximal torus of $G_{[a]}(t)$. On the other hand, there is a natural decomposition of the torus $T_\mu$ as follows:
 \[
\begin{matrix}
T_\mu= T_{\mu_1}\times\cdots\times T_{\mu_k}

\end{matrix}
\]
where $T_{\mu_i}^F\cong \GGL_{1}(\bb{F}_{q^{\mu_i}})$. Clearly, for each $\mu_i$, there is a unique index $[a]$ such that $T_{\mu_i}\subseteq T_{\mu[t,[a]]}$, and
 \[
 T_{\mu[t,[a]]}=\prod_{i}T_{\mu_i}
 \]
where the product runs over the integers $i\in\{1,\cdots,k\}$ such that $T_{\mu_i}\subseteq T_{\mu[t,[a]]}$. In other words, we get a map $f_t$ from the set $\{\mu_i\}$ to the set $\{[a]\}$, and then we can regard the index $\mu[t,[a]]$ as a partition $\mu[t,[a]]=(\mu_{j})\subset \mu$ where $\mu_j\in\{\mu_i\} $ with $f_t(\mu_j)=[a]$. If $[a]$ is not in the image of $f_t$, then we set $\mu[t,[a]]=0$. Hence, we have $\#[a]\big|\mu[t,[a]]$. For abbreviation, we write $\mu[t,a]$ instead of $\mu[t,[a]]$. It is easy to check that $\mu[t,a]=\mu[t',a]$ for any $[a]$ if and only if $t$ and $t'$ are in the same $W_G(T_\mu)^F$-conjugacy class.

 Let $\mu'$ and $\mu''$ be two partitions such that $|\mu'|+|\mu''|=n$ and $\mu=\mu'\bigcup\mu''$. We fix an embedding
 \begin{equation}\label{e1gl}
 T_{\mu}=T_{\mu'}\times T_{\mu''}\hookrightarrow \GGL_{|\mu'|}\times \GGL_{|\mu''|}\hookrightarrow \GGL_n.
 \end{equation}
 It is easy to check that
\[
w_\mu=\prod_{i=1}^{r}\prod_{\ell=1}^{k_i}\langle c_{i,\ell}\rangle\subset \GGL_{|\mu'|}\times \GGL_{|\mu''|},
\]
 which implies that
  \begin{equation}\label{cw2}
 C_{W,F}(w_\mu)=\left\langle C_{W,F}(w_\mu)\bigcap \GGL_{|\mu'|},C_{W,F}(w_\mu)\bigcap \GGL_{|\mu''|},S \right\rangle
 \end{equation}
 where $S$ is a certain subgroup of $\psi(\BB{S}_{k_1}\times\cdots\times \BB{S}_{k_r})$ such that $S \bigcap \GGL_{|\mu'|}=S \bigcap \GGL_{|\mu''|}=1$.
Note that $W_{\GGL_n}(T_{\mu})^F= C_{W,F}(w_\mu)$, $W_{\GGL_{|\mu'|}}(T_{\mu'})^F=C_{W,F}(w_\mu)\bigcap \GGL_{|\mu'|}\fq$, and $W_{\GGL_{|\mu''|}}(T_{\mu''})^F=C_{W,F}(w_\mu)\bigcap \GGL_{|\mu''|}\fq$. We get
 \begin{equation}
 |W_{\GGL_n}(T_{\mu})^F|=| C_{W,F}(w_\mu)|=C_{\mu,\mu'}\cdot|W_{\GGL_{|\mu'|}}(T_{\mu'})^F|\cdot|W_{\GGL_{|\mu''|}}(T_{\mu''})^F|.
 \end{equation}

From now on, we fix a $t\in T_{\mu}^F$. We now turn to consider the set
\[
D(\mu,\mu',t):=\{t'\in T_{\mu'}^F\textrm{ such that }t'=^{w'}t|_{T_{\mu'}}\textrm{ for some } w'\in W_G(T_{\mu})^F\}.
 \]
If $t'\in D(\mu,\mu',t)$, then it is easy to check that
\begin{itemize}

\item  ${}^{w'}t'\in D(\mu,\mu',t)$ for each $w'\in W_{\GGL_{|\mu'|}}(T_{\mu'})^F$;

\item  $\mu'[t',a]\subset \mu[t,a]$ for each $[a]$;

\item $\mu'[t',a]=\mu'[t'',a]$ for any $[a]$ if and only if $t'$ and $t''$ are in the same $W_{\GGL_{|\mu'|}}(T_{\mu'})^F$-conjugacy class.
\end{itemize}
Hence, we get an injection from the set of $W_{\GGL_{|\mu'|}}(T_{\mu'})^F$-conjugacy classes in $D(\mu,\mu',t)$ to the set
\begin{equation}\label{glp}
P(\mu,\mu',t):=\left\{f:A\to P\textrm{ such that }\#[a]\big|f([a]),\ \bigcup_{[a]}f([a])=\mu',\textrm{ and }f([a])\subset \mu[t,a]\right\}
\end{equation}
where $A$ is the set of indexes $[a]$, and $P$ is the set of partitions of any integers. To show this is actually a bijection, we shall associate every map $f \in P(\mu,\mu',t)$ to a semisimple $t'\in D(\mu,\mu',t)$ as follows. We have a decomposition of $T_{\mu'}$:
\[
T_{\mu'}=\prod_{[a]}T_{f([a])}
\]
where $T_{f([a])}$ is an $F$-stable maximal torus of $\GGL_{|f([a])|}$. Let $T_{0,f([a])}$ be the $F$-stable maximal torus of $\GGL_{|f([a])|}$ consisting of diagonal matrices. Put
\[
t_{0,f([a])}=\rm{diag}\left(a,a^q,a^{q^2}\cdots,a^{q^{|f([a])|-1}}\right)\in T_{0,f([a])}.
\]
Let $w_{f([a])}$ be the element of the Weyl group of $\GGL_{|f([a])|}$ corresponding to the partition $f[a]$. Since $\#[a]\big|f([a])$ and $a^{q^{\#[a]}}=a$, we have $F(w_{f([a])}\cdot t_{0,f([a])})=t_{0,f([a])}$, or equivalently $t_{0,f([a])}\in T_{0,f([a])}^{w_{f([a])}F}$. Let $g\in \GGL_{|f([a])|}$ such that the image of $gF(g^{-1})$ in the Weyl group is $w_{f([a])}$. Then $t_{f([a])}:={}^{g^{-1}}t_{0,f([a])}\in T_{f([a])}^F$. Put
\[
t':=\prod_{[a]}t_{f([a])}.
\]
 Clearly, $t'\in D(\mu,\mu',t)$ and $\mu[t',a]=f([a])$ for each $[a]$.
\begin{lemma}\label{2gl}
Keep the notations and assumptions as above, there is a bijection from the set $P(\mu,\mu',t)$ to the $W_{\GGL_{|\mu'|}}(T_{\mu'})^F$-conjugacy classes in $D(\mu,\mu',t)$.
\end{lemma}

Set
\[
M(\mu,\mu',t,t'):=\{w'\in W_{\GGL_{n}}(T_{\mu})^F: t'=^{w'}t|_{T_{\mu'}} \}.
\]
Recall that the action of $W_{\GGL_{n}}(T_{\mu'})^F$ on $T_{\mu'}$ is equal to the action of $C_{W,F}(w_\mu)$ on $T_0^{w_{\mu}F}$. By our discussion of the action of $C_{W,F}(w_\mu)$ in (\ref{cw}) and (\ref{cw2}), we have
\begin{equation}\label{2glm}
\#M(\mu,\mu',t,t')=\left\{
  \begin{aligned}
 &|W_{\GGL_{|\mu'|}}(T_{\mu'},t')^F|\cdot|W_{\GGL_{|\mu''|}}(T_{\mu''})^F|\cdot\prod_{[a]}C_{\mu[t,a],\mu'[t',a]}&\textrm{ if }t'\in D(\mu,\mu',t) ;\\
 &0&\textrm{ otherwise}
\end{aligned}
\right.
\end{equation}
where $W_{G}(T,t):=\{w\in W_{G}(T)|{}^wt=t\}$.

\subsection{The unitary groups}\label{sec3.2}
In this subsection, we assume that $G=\UU_n$. We may choose the hermitian form that defines the unitary group to be represented by the identity matrix so that
\[
\UU_n\fq =\{g \in \GGL_n(\overline{\bb{F}}_q) | g \cdot {}^t\bar{g}=I_n\},
\]
where $\bar{g}$ is obtained from raising matrix entries in $g$ to the $q$-th power.
Then the action of the Frobenius endomorphism $F$ is given by
$
g\mapsto {}^t\bar{g}^{-1}.
$

Take an $F$-stable maximal torus $T_0$ of $G$ that is contained in an $F$-stable Borel subgroup. As before, we have $W=W_G(T_0)\cong \BB{S}_n$. It is easy to check that with this identification, the action of $F$ on $W$ is given by
\[
F(w)=(w_0ww_0)
\]
where $w_0$ is the longest element in $\BB{S}_n$.
Note that $w$ and $w'$ are $F$-conjugate if and only if $ww_0$ and $w'w_0$ are conjugate. So the $F$-conjugacy classes in $W$ are parametrized by partitions of $n$. We write $T_{\mu}$ instead of $T_{w}$ if $w$ is in the $F$-conjugacy class corresponding to $\mu$ or equivalently $ww_0$ is in the $W$-conjugacy class corresponding to $\mu$.

Let $\mu'$ and $\mu''$ be two partitions such that $|\mu'|+|\mu''|=n$. Suppose that $\mu=\mu'\bigcup \mu''$. Then we have
 \begin{equation}\label{uw}
 |W_{\UU_n}(T_{\mu})^F|=C_{\mu,\mu'}\cdot|W_{\UU_{|\mu'|}}(T_{\mu'})^F|\cdot|W_{\UU_{|\mu''|}}(T_{\mu''})^F|.
 \end{equation}
Similar to subsection \ref{sec3.1}, we define $\mu[t,a]$, $D(\mu,\mu',t)$ and $P(\mu,\mu',t)$.
\begin{lemma}\label{2u}
Keep the notations and assumptions as above, there is a bijection from $P(\mu,\mu',t)$ to the $W_{\UU_{|\mu'|}}(T_{\mu'})^F$-conjugacy classes in $D(\mu,\mu',t)$.
\end{lemma}
\begin{proof}
The proof is similar to Lemma \ref{2gl}. The only difference between the unitary groups case and the general linear groups case is the construction of $t_{0,f([a])}$. In the unitary groups case, we take $t_{0,f([a])}=\rm{diag}\left(a,a^{-q},a^{{-q}^2}\cdots,a^{{-q}^{|f([a])|-1}}\right)$ instead of $\rm{diag}\left(a,a^{q},a^{q^2}\cdots,a^{q^{|f([a])|-1}}\right)$, and that is why the definitions of $[a]$ differ between these two cases.
\end{proof}
As before, we set
\[
M(\mu,\mu',t,t'):=\{w'\in W_{\UU_{n}}(T_{\mu})^F : t'=^{w'}t|_{T_{\mu'}} \}.
\]
Then
\begin{equation}\label{um}
\#M(\mu,\mu',t,t')=\left\{
  \begin{aligned}
 &|W_{\UU_{|\mu'|}}(T_{\mu'},t')^F|\cdot|W_{\UU_{|\mu''|}}(T_{\mu''})^F|\cdot\prod_{[a]}C_{\mu[t,a],\mu'[t',a]}&\textrm{ if }t'\in D(\mu,\mu',t) ;\\
 &0&\textrm{ otherwise}
\end{aligned}
\right.
\end{equation}

\subsection{The symplectic and odd special orthogonal groups}\label{sec3.3}
Let G be an odd special orthogonal group $G=\so_{2n+1}$  or a symplectic group $\sp_{2n}$. The Weyl group $W$ of $G$ is isomorphic to the semidirect product of $\BB{S}_n$ with an abelian normal subgroup of order $2^n$. Clearly, $F$ acts trivially on $W$. Then the conjugacy classes of $W$ are equal to the $F$-conjugacy classes of $W$. It is well known that the conjugacy classes of $W$ are parameterized by bipartitions $(\mu,\lambda)$, where $|\mu|+|\lambda| = n$. We allow $\mu = 0$ or $\lambda = 0$. For abbreviation, we write $T_{\mu,\lambda}$ instead of $T_{w_{\mu,\lambda}}$.

The canonical form of an element $w_{\mu,\lambda} \in W$ can be described as follows. Suppose that $\mu=(\mu_1,\cdots,\mu_k)$, $\lambda=(\lambda_{1},\cdots,\lambda_{l})$. Put
\[
N_{\mu_i}:= \mu_1 + \cdots+ \mu_{i}\textrm{ and }N_{\lambda_j}:= |\mu|+\lambda_1 + \cdots+ \lambda_{j}.
 \]
 Let $T_0$ be the group of diagonal matrices.
Let $t = (x_1, \cdots , x_n) \in T_0$ and $w_{\mu,\lambda} \cdot t =t'= (x'_1, \cdots , x'_n)\in T_0$.
Then
\begin{equation}\label{act1}
(x'_{N_{\mu_i}+1},x'_{N_{\mu_i}+2},\cdots,x'_{N_{\mu_{i+1}}})=(x_{N_{\mu_{i+1}}},x_{N_{\mu_i}+1},x_{N_{\mu_i}+2},\cdots,x_{N_{\mu_{i+1}}-1})
\end{equation}
and
\begin{equation}\label{act2}
(x'_{N_{\lambda_j}+1},x'_{N_{\lambda_j}+2},\cdots,x'_{N_{\lambda_{j+1}}})=((x_{N_{\lambda_{j+1}}})^{-1},x_{N_{\lambda_j}+1},x_{N_{\lambda_j}+2},\cdots,x_{N_{\lambda_{j+1}}-1}).
\end{equation}

 From now on, consider the symplectic group $G=\sp_{2n}$.
Let $t\in T_{\mu,\lambda}^F$ be a semisimple element. There is a natural embedding:
  \[
\begin{matrix}
T_{\mu,\lambda}= \prod_{[a]}T_{\mu[t,[a]],\lambda[t,[a]]}\hookrightarrow  C_{G^F}(t)= \prod_{[a]}G_{[a]}(t).
\end{matrix}
\]
where $T_{\mu[t,[a]],\lambda[t,[a]]}=T_{\mu,\lambda}\bigcap G_{[a]}(t)$ is an $F$-stable maximal torus of $G_{[a]}(t)$. On the other hand, there is a natural decomposition of the torus $T_{\mu,\lambda}$ as follows:
 \[
\begin{matrix}
T_{\mu,\lambda}= T_{\mu_1}\times\cdots\times T_{\mu_k}\times T_{\lambda_1}\times\cdots\times T_{\lambda_l}
\end{matrix}
\]
where $T_{\mu_i}^F\cong \GGL_{1}(\bb{F}_{q^{\mu_i}})$ and $T_{\lambda_i}^F\cong \UU_{1}(\bb{F}_{q^{\lambda_i}})$. Clearly, for each $\mu_i$ (resp. $\lambda_j$), there is a unique index $[a]$ such that $T_{\mu_i}\subseteq G_{[a]}(t)$ (resp. $T_{\lambda_j}\subseteq G_{[a]}(t)$), and
 \[
 T_{\mu[t,[a]],\lambda[t,[a]]}=  T_{\mu[t,[a]]}\times  T_{\lambda[t,[a]]}=\prod_{i}T_{\mu_i}\times \prod_{j}T_{\lambda_j}
 \]
where the product of $i$ (resp. $j$) runs over the integers $i\in\{1,\cdots,k\}$ (resp. $j\in\{1,\cdots,l\}$) such that $T_{\mu_i}\subseteq G_{[a]}(t)$ (resp. $T_{\lambda_j}\subseteq G_{[a]}(t)$). In other words, we get a map $f_t$ from the set $\{\mu_i\}\cup \{\lambda_j\}$ to the set $\{[a]\}$, and then we can regard the index $(\mu[t,[a]],\lambda[t,[a]])$ as a bipartition $(\mu[t,[a]],\lambda[t,[a]])=((\mu_{i'}),(\lambda_{j'}))\subset (\mu,\lambda)$ where $\mu_{i'}\in\{\mu_i\} $ with $f_t(\mu_{i'})=[a]$ and $\lambda_{j'}\in\{\lambda_j\} $ with $f_t(\lambda_{j'})=[a]$. If $[a]$ is not in the image of $f_t$, then we set $(\mu[t,[a]],\lambda[t,[a]])=0$. Hence, we have
\[
\#[a]\big|\mu[t,[a]]\ \textrm{ and }\ \#[a]\big|2\lambda[t,[a]].
 \]
 For abbreviation, we write $\mu[t,a]$ (resp. $\lambda[t,a]$) instead of $\mu[t,[a]]$ (resp. $\lambda[t,[a]]$). It is easy to check that $(\mu[t,a],\lambda[t,a])=(\mu[t',a],\lambda[t',a])$ for any $[a]$ if and only if $t$ and $t'$ are in the same $W_G(T_{\mu,\lambda})^F$-conjugacy class.

 Let $(\mu',\lambda')$ and $(\mu'',\lambda'')$ be two partitions such that $\mu=\mu'\bigcup\mu''$ and $\lambda=\lambda'\bigcup\lambda''$. Then we fix an embedding
 \begin{equation}\label{e1}
 T_{\mu,\lambda}=T_{\mu',\lambda'}\times T_{\mu'',\lambda''}\hookrightarrow \sp_{2(|\mu'|+|\lambda'|)}\times \sp_{2(|\mu''|+|\lambda''|)}\hookrightarrow \sp_{2n}.
 \end{equation}
In a similar way to the general linear groups case, we have
 \begin{equation}\label{cw3}
 |W_{G}(T_{\mu,\lambda})^F|\cong |C_{W,F}(w_{\mu,\lambda})|=C_{\mu,\mu'}\cdot C_{\lambda,\lambda'}\cdot|W_{\sp_{2(|\mu'|+|\lambda'|)}}(T_{\mu',\lambda'})^F|\cdot|W_{\sp_{2(|\mu''|+|\lambda''|)}}(T_{\mu'',\lambda''})^F|.
 \end{equation}

From now on, we fix a $t\in T_{\mu}^F$. With the embedding (\ref{e1}), we now turn to describe the set
\[
D(\mu,\lambda,\mu',\lambda',t):=\{t'\in T_{\mu',\lambda'}^F\textrm{ such that }t'=^{w'}t|_{T_{\mu',\lambda'}}\textrm{ for some } w'\in W_{\sp_{2n}}(T_{\mu,\lambda})^F\}.
 \]
It is easily seen that
\begin{itemize}

\item  ${}^{w'}t'\in D(\mu,\lambda,\mu',\lambda',t)$ for each $w'\in W_{\sp_{2(|\mu'|+|\lambda'|)}}(T_{\mu',\lambda'})^F$.

\item  $\mu'[t',a]\subset \mu[t,a]$ for each $[a]$.

\item  $\lambda'[t',a]\subset \lambda[t,a]$ for each $[a]$.

\item $(\mu[t',a],\lambda[t',a])=(\mu[t'',a],\lambda[t'',a])$ for any $[a]$ if and only if $t'$ and $t''$ are in the same $W_{\sp_{2(|\mu'|+|\lambda'|)}}(T_{\mu',\lambda'})^F$-conjugacy class.
\end{itemize}
Hence, we get an injection from the set of $ W_{\sp_{2(|\mu'|+|\lambda'|)}}(T_{\mu',\lambda'})^F$-conjugacy classes in $D(\mu,\lambda,\mu',\lambda',t)$ to the set
\[
P(\mu,,\lambda,\mu',\lambda',t):=\left\{(f, f'):A\times A\to P\times P\ \Big|
{\mbox{$\begin{array}{c}\bigcup_{[a]}f([a])=\mu',\ \bigcup_{[a]}f'([a])=\lambda'\\
\textrm{and }f([a])\subset \mu[t,a],\ f'([a])\subset \lambda[t,a]\\
\end{array}$}}\right\}
\]
where $A$ is the set of indexes $[a]$ and $P$ is the set of partitions of any integers.

\begin{lemma}\label{2sp}
Keep the notations and assumptions as above, there is a bijection from $P(\mu,\lambda,\mu',\lambda',t)$ to the $W_{\sp_{2(|\mu'|+|\lambda'|)}}(T_{\mu',\lambda'})^F$-conjugacy classes in $D(\mu,\lambda,\mu',\lambda',t)$.
\end{lemma}
\begin{proof}
The proof is similar to Lemma \ref{2gl}. To show this injection is actually a bijection, we attach to a pair of map $(f, f') \in P(\mu,,\lambda,\mu',\lambda',t)$ a semisimple $t'\in D(\mu,\lambda,\mu',\lambda',t)$ as follows. Put
\[
G_{f,f',[a]}:=\left\{
  \begin{array}{ll}
\GGL_{|f([a])|+|f'([a])|}&\textrm{ if $\#[a]$ is odd};\\
\UU_{|f([a])|+|f'([a])|}&\textrm{ if $\#[a]$ is even}.
\end{array}
\right.
\]
Assume that $f([a])=(f([a])_1,\cdots,f([a])_{k'})$ and $f'([a])=(f'([a])_1,\cdots,f'([a])_{l'})$. We have a decomposition of $T_{\mu',\lambda'}$:
\[
T_{\mu',\lambda'}=\prod_{[a]}\left(T_{f([a])}\times  T_{f'([a])}\right)=\prod_{[a]}\left(\prod_{i}T_{f([a])_i}\times  \prod_{j}T_{f'([a])_j}\right)
\]
where $T_{f([a])}\times T_{f'([a])}$ is an $F$-stable maximal torus of $G_{f,f',[a]}$, and the $F$-rational points of the torus $T_{f([a])}=\prod_{i}T_{f([a])_i}$ (resp. $T_{f'([a])}=\prod_{j}T_{f'([a])_j}$) are of the form:
\[
 \prod_{i}\GGL_{1}(\bb{F}_{q^{f([a])_i}}) \quad \left(\textrm{ resp. $ \prod_{j}\UU_{1}(\bb{F}_{q^{f'([a])_j}})$}\right).
  \]
  If $G_{f,f''[a]}=\GGL_{|f([a])+f'([a])|}$, We take a semisimple $t_{f,f',[a]}\in T_{f([a])}^F\times T_{f'([a])}^F$ as in the general linear groups case.

Assume that $G_{f,f''[a]}=\UU_{|f([a])+f'([a])|}$. Since the torus $T_{f([a])}$ is a product of $\GGL_{1}(\bb{F}_{q^{f([a])_i}})$, we can take a semisimple $t_{f[a]}\in T_{f([a])}^F$ similarly to the general linear groups case. Consider the natural embedding:
\[
T_{f'([a])_j}\hookrightarrow \UU_{f'([a])_j}.
\]
With above embedding, the torus $T_{f'([a])_j}$ becomes an $F$-stable maximal torus of $\UU_{f'([a])_j}$, and then $T_{f'([a])}$ is corresponding to $w_{f'[a]}$ where $w_{f'[a]}$ is in the Weyl group of $\UU_{f'([a])}$. Let $T_{0,f'([a])_j}$ be the $F$-stable maximal torus of $\UU_{f'([a])_j}$ consisting of diagonal matrices. Put
\[
t_{0,f'([a]_j)}=\rm{diag}\left(a,a^q,a^{q^2}\cdots,a^{q^{f([a])_j-1}}\right)\in T_{0,f'([a])}.
\]
Recall that $\#[a]$ is even, and $a^{-q^{\frac{\#[a]}{2}}}=a$. It follows from $\#[a]\big|2\lambda[t,a]$ and $f'[a]\subset \lambda[t,a]$ that $\#[a]\big|2f'(a)_j$ and $a^{-q^{f'(a)_j}}=a$. Then
\[
\begin{aligned}
w_{f'[a]}\cdot F (t_{0,f'([a]_j)})=&w_{f'[a]}\cdot\rm{diag}\left(a^q,a^{q^2}\cdots,a^{q^{f([a])_j-1}},a^{q^{f([a])_j}}\right)\\
=&\rm{diag}\left(a^{-q^{f([a])_j}},a^q,a^{q^2}\cdots,a^{q^{f([a])_j-1}}\right)\\
=&\rm{diag}\left(a,a^q,a^{q^2}\cdots,a^{q^{|f([a])_j|-1}}\right)\\
=& t_{0,f'([a]_j)}.
\end{aligned}
\]
 Let $g\in \UU_{f'([a])_j}$ such that the image of $gF(g^{-1})$ in the Weyl group is $w_{f'([a])}$. Then $F(t_{f'([a])_j})=t_{f'([a])_j}$ where $t_{f'([a])_j}:={}^{g^{-1}}t_{0,f([a]_j)}\in T_{f'([a]_j)}$. Therefore we find a semisimple element
\[
t_{f,f',[a]}:=t_{f[a]}\times\prod_{j}t_{f([a]_j)}
\]
which plays the same role as $t_{f[a]}$ in the proof of Lemma \ref{2gl}.

Put
\[
t':=\prod_{[a]}t_{f,f',[a]}.
\]
 Clearly, $t'\in D(\mu,\lambda,\mu',\lambda',t)$, $\mu[t',a]=f([a])$, and $\lambda[t',a]=f'([a])$ for each $[a]$.

\end{proof}

Set
\[
M(\mu,\lambda,\mu',\lambda',t,t'):=\{w'\in W_{\sp_{2n}}(T_{\mu,\lambda})^F: t'=^{w'}t|_{T_{\mu',\lambda'}} \}.
\]
Recall that the action of $W_{\sp_{2n}}(T_{\mu,\lambda})^F$ on $T_{\mu',\lambda'}$ is equal to the action of $C_{W,F}(w_{\mu,\lambda})$ on $T_0^{w_{\mu,\lambda}F}$. By our discussion of the action of $C_{W,F}(w_{\mu,\lambda})$ in (\ref{cw3}, we have
\begin{equation}\label{spm}
\begin{aligned}
&\#M(\mu,\lambda,\mu',\lambda',t,t')\\
=&\left\{
  \begin{aligned}
 &|W_{\sp_{2n'}}(T_{\mu',\lambda'},t')^F|\cdot|W_{\sp_{2n''}}(T_{\mu'',\lambda''})^F|\cdot\prod_{[a]}C_{\mu[t,a],\mu'[t',a]}\cdot C_{\lambda[t,a],\lambda'[t',a]}&\textrm{ if }t'\in D(\mu,\lambda,\mu',\lambda',t) ;\\
 &0&\textrm{ otherwise}
\end{aligned}
\right.
\end{aligned}
\end{equation}
where $n'=|\mu'|+|\lambda'|$, $n''=|\mu''|+|\lambda''|$.

\subsection{The even special orthogonal groups}\label{sec3.4}

Let $G=\so^+_{2n}$. The Weyl group $W$ of $G$ is isomorphic to a subgroup group of the semidirect product of $\BB{S}_n$ with an abelian normal subgroup of order $2^n$ acting on $T_0$. More precisely, an element of $W$ sends $(x_1, \cdots , x_n)$ to $(x_1^{\pm1},\cdots, x_n^{\pm1})$ such that there are even $-1$ in $(x_1^{\pm1},\cdots, x_n^{\pm1})$. Clearly, $F$ acts trivially on $W$. Then the conjugacy classes of $W$ are equal to the $F$-conjugacy classes of $W$.

It is well known that the conjugacy classes of $W$ correspond to the bipartitions $(\mu,\lambda)$ of $n$ where $l$ is even, and $\mu=(\mu_1,\cdots,\mu_k)$, $\lambda=(\lambda_{1},\cdots,\lambda_{l})$ with $l$ even. However, it is not a $1-1$ correspondence. If $l=0$ and $\mu_i$ is even for each $i$, then there are two conjugacy classes of $W$ corresponding to $(\mu,\lambda)$. If not, then this correspondence is $1-1$. For a bipartition $(\mu,\lambda)$, if $l\ne0$ or $\mu_i$ is odd for some $i$, we denote the canonical form of an element corresponding it by $w_{\mu,\lambda}$. And if $l=0$ and $\mu_i$ is even, we denote the canonical form by $w_{\mu,\lambda}^\pm$. Since our following calculation process is the same for both $w_{\mu,\lambda}^+$ and $w_{\mu,\lambda}^-$ , for abuse of notations, we write $w_{\mu,\lambda}$ instead of $w_{\mu,\lambda}^+$, and the result of $w_{\mu,\lambda}^-$ is the same as $w_{\mu,\lambda}^+$.

The action of the canonical form of the element $w_{\mu,\lambda}$ on $T_0$ is the same as in subsection \ref{sec3.3}. Let $t = (x_1, \cdots , x_n) \in T_0$. Then we have a natural decompositions:
\[
w_{\mu,\lambda}=\prod_{i=1}^{k}w_{\mu_i}\times\prod_{j=1}^lw_{\lambda_{j}}
\]
and
\[
t=\prod_{i=1}^{k}t_{\mu_i}\times\prod_{j=1}^lt_{\lambda_{j}}
\]
where $w_{\mu_i}$ and $w_{\lambda_j}$ act on $t_{\mu_i}$ and $t_{\lambda_j}$ as in (\ref{act1}) and (\ref{act2}), respectively. Let $w^-_{i'}$ with $1\le i'\le k$ be a element of $W$ such that
\[
w^-_{i'}\cdot t=\prod_{i=1}^{i'-1}t_{\mu_i}\times t_{\mu_{i'}}^{-1}\times\prod_{i=i'+1}^{k}t_{\mu_i}\times\prod_{j=1}^lt_{\lambda_{j}}.
\]
Let
\[
W_{\mu,\lambda}:=\left\langle\prod_{i=1}^{k}\langle w_{\mu_i},w_i^{-}\rangle,\prod_{i=1}^{l}\langle w_{\lambda_i}\rangle\right\rangle\bigcap W.
\]
Set
\[
\begin{aligned}
S_\mu:=\left\{\sigma\in\BB{S}_{|\mu|} \big| \mu_{\sigma(i)}=\mu_i \textrm{ for each }i\right\},\\
S_\lambda:=\left\{\tau\in\BB{S}_{|\lambda|} \big| \lambda_{\tau(j)}=\lambda_j \textrm{ for each }j\right\},\\
\end{aligned}
\]
and
\[
S:=\left\{w\in W \big| w\cdot t=\prod_{i=1}^{k}t_{\mu_{\sigma(i)}}\times\prod_{j=1}^lt_{\lambda_{\tau(j)}}\right\}.
\]
 Then  $S\cong S_\mu\times S_\lambda$, and
 \begin{equation}\label{cw33}
 C_{W,F}(w_{\mu,\lambda})=\left\langle S,W_{\mu,\lambda}\right\rangle.
 \end{equation}
Suppose that $\mu=(\mu_1,\cdots,\mu_k)=(e_1^{k_1},e_2^{k_2},\cdots,e_r^{k_r})$ and $\lambda=(\lambda_1,\cdots,\lambda_l)=(f_1^{l_1},f_2^{l_2},\cdots,f_m^{l_{m}})$. We have
 \[
 \# C_{W,F}(w_{\mu,\lambda}))=
 \left\{
  \begin{array}{ll}
 \left(\prod_{i=1}^{r}(2e_i)^{k_i}k_i!\right)&\textrm{ if $l=0$ and $n_i$ is even};\\
  \frac{1}{2}\cdot\left(\prod_{i=1}^{r}(2e_i)^{k_i}k_i!\right)\cdot\left(\prod_{i=1}^{m}(2f_i)^{l_i}l_i!\right)&\textrm{ otherwise}.
\end{array}
\right.
 \]

Let $(\mu',\lambda')$ and $(\mu'',\lambda'')$ be two partitions such that $|\mu'|+|\lambda'|+|\mu''|+|\lambda''|=n$, $\mu=\mu'\bigcup\mu''$, and $\lambda=\lambda'\bigcup\lambda''$. There is a embedding
 \begin{equation}\label{e2}
 T_{\mu,\lambda}=T_{\mu',\lambda'}\times T_{\mu'',\lambda''}\hookrightarrow \so^\epsilon_{2(|\mu'|+|\lambda'|)}\times \so^{\epsilon'}_{2(|\mu''|+|\lambda''|)}\hookrightarrow \so^+_{2n}
 \end{equation}
 for some $\epsilon,\epsilon'\in\{\pm\}$.
Let
\[
\BB{S}_{\mu',\lambda'}:=S\bigcap (\so^\epsilon_{2(|\mu'|+|\lambda'|)}\times I)
\]
and
\[
\BB{S}_{\mu'',\lambda''}:=S\bigcap (I\times \so^{\epsilon'}_{2(|\mu''|+|\lambda''|)})
\]
where $I$ is the identity in certain groups.
Then
 \begin{equation}
 |W_{\so^+_{2n}}(T_{\mu,\lambda})^F|=C_{\mu,\mu'}\cdot C_{\lambda,\lambda'}\cdot|\BB{S}_{\mu',\lambda'}|\cdot|\BB{S}_{\mu'',\lambda''}|\cdot |W_{\mu,\lambda}|.
 \end{equation}

 Let $g\in G$ such that the image of $gF(g^{-1})$ in the Weyl group $W$ is $w_{\mu,\lambda}$. Let
 \[
S_{\mu,\lambda,\mu',\lambda'}:= \left\langle {}^{g^{-1}}\BB{S}_{\mu',\lambda'},{}^{g^{-1}}\BB{S}_{\mu'',\lambda''},{}^{g^{-1}}W_{\mu,\lambda}\right\rangle
 \]
 be a subgroup of $W_G(T_{\mu,\lambda})$.
We set $T_{\mu_i}$, $T_{\lambda_i}$, $f_t$, $\mu[t,a]$ and $\lambda[t,a]$ as in subsection \ref{sec3.3}. Clearly, $\mu[t,a]=\mu[t',a]$ and $\lambda[t,a]=\lambda[t',a]$ for any $[a]$ if and only if $t$ and $t'$ are in the same $S_{\mu,\lambda,\mu',\lambda'}$-conjugacy class.
We fix a $t\in T_{\mu,\lambda}^F$.
We now turn to describe the set
\[
D^+(\mu,\lambda,\mu',\lambda',t):=\{t'\in T_{\mu',\lambda'}^F\textrm{ such that }t'=^{w'}t|_{T_{\mu',\lambda'}}\textrm{ for some } w'\in W_{\so^+_{2n}}(T_{\mu,\lambda})^F\}.
 \]
 With the embedding (\ref{e2}), the group $S_{\mu,\lambda,\mu',\lambda'}$ sends $T_{\mu',\lambda'}^F\times I$ to itself. So we define the action of $S_{\mu,\lambda,\mu',\lambda'}$ on $T_{\mu',\lambda'}^F$ as its action on $T_{\mu',\lambda'}^F\times I$.

\begin{lemma}\label{2so}
Keep the notations and assumptions as above, there is a bijection from $P(\mu,\lambda,\mu',\lambda',t)$ to $S_{\mu,\lambda,\mu',\lambda'}$-conjugacy classes in $D^+(\mu,\lambda,\mu',\lambda',t)$.
\end{lemma}
\begin{proof}
The proof is similar to Lemma \ref{2sp}.
\end{proof}

Set
\[
M^+(\mu,\lambda,\mu',\lambda',t,t'):=\{w'\in W_{\so^+_{2n}}(T_{\mu,\lambda})^F: t'=^{w'}t|_{T_{\mu',\lambda'}} \}.
\]
Recall that the action of $W_{\so^+_{2n}}(T_{\mu,\lambda})^F$ on $T_{\mu',\lambda'}$ is equal to the action of $C_{W,F}(w_{\mu,\lambda})$ on $T_0^{w_{\mu,\lambda}F}$. By our discussion of the action of $C_{W,F}(w_{\mu,\lambda})$, we have
\begin{lemma}\label{som}
Keep the notations and assumptions as above. Then the following hold.
\begin{itemize}

\item There is a bijection from $P(\mu,\lambda,\mu',\lambda',t)$ to the $S_{\mu,\lambda,\mu',\lambda'}$-conjugacy class in $D^+(\mu,\lambda,\mu',\lambda',t)$;

 \item  We have

\begin{equation}\label{sop}
 \begin{aligned}
&\#M^+(\mu,\lambda,\mu',\lambda',t,t')\\
=&\left\{
  \begin{aligned}
 &|\BB{S}_{\mu',\lambda'}|\cdot|\BB{S}_{\mu'',\lambda''}|\cdot |W_{\mu,\lambda,{}^gt'}|\cdot\prod_{[a]}C_{\mu[t,a],\mu'[t',a]}\cdot C_{\lambda[t,a],\lambda'[t',a]}&\textrm{ if }t'\in D^+(\mu,\lambda,\mu',\lambda',t) ;\\
 &0&\textrm{ otherwise}
\end{aligned}
\right.
 \end{aligned}
\end{equation}
where $W_{\mu,\lambda,{}^gt'}:=\{w\in W_{\mu,\lambda}|{}^w({}^gt\times I)={}^gt\times I\}$.
\item
\[
|W_{\so^+_{2n}}(T_{\mu,\lambda})^F|=C_{\mu,\mu'}\cdot C_{\lambda,\lambda'}\cdot|\BB{S}_{\mu',\lambda'}|\cdot|\BB{S}_{\mu'',\lambda''}|\cdot |W_{\mu,\lambda}|.
\]
\end{itemize}

\end{lemma}

Now consider $G = \so_{2n}^-$. The Weyl group $W$ is isomorphic to the Coxeter group $W(D_n) $ of type $D_n$. Let $W(B_n) $ be the Coxeter group of type $B_n$. Then there is a natural injective form $W(D_n)$ to $W(B_n)$. The action of $F$ on $W\cong W(D_n)$ is known to be realized as the conjugation by some element $a\in W(B_n)\setminus W(D_n)$, see \cite[section 11]{St} where the action can be explicitly described. Then the $F$-conjugacy classes of $W(D_n)$ is of shape $\{x^{-1}waxa^{-1} : x \in W(D_n)\}$. By \cite[3.3.6]{C}, $W_G(T_w)^F\cong C_{W(D_n)}(wa)$.

Let $t\in T_{\mu,\lambda }^F$. Let $\mu=\mu'\bigcup \mu''$ and $\lambda=\lambda'\bigcup\lambda''$.
We set $S_{\mu,\lambda,\mu',\lambda'}$, $D^-(\mu,\lambda,\mu',\lambda',t)$ and $M^-(\mu,\lambda,\mu',\lambda',t,t')$ as before. Then the following hold.
\begin{lemma}\label{so2}
Keep the notations and assumptions as above. Then the following hold.
\begin{itemize}

\item There is a bijection from the set $P(\mu,\lambda,\mu',\lambda',t)$ to the $S_{\mu,\lambda,\mu',\lambda'}$-conjugacy classes in the set $D^-(\mu,\lambda,\mu',\lambda',t)$;

 \item  We have
\begin{equation}\label{som2}
  \begin{aligned}
&\#M^-(\mu,\lambda,\mu',\lambda',t,t')\\
=&\left\{
  \begin{aligned}
 &|\BB{S}_{\mu',\lambda'}|\cdot|\BB{S}_{\mu'',\lambda''}|\cdot |W_{\mu,\lambda,{}^gt'}|\cdot\prod_{[a]}C_{\mu[t,a],\mu'[t',a]}\cdot C_{\lambda[t,a],\lambda'[t',a]}&\textrm{ if }t'\in D^-(\mu,\lambda,\mu',\lambda',t) ;\\
 &0&\textrm{ otherwise}
\end{aligned}
\right.
 \end{aligned}
\end{equation}
where $W_{\mu,\lambda,{}^gt'}:=\{w\in W_{\mu,\lambda}|{}^w({}^gt\times I)={}^gt\times I\}$ and $gF(g^{-1})=w_{\mu,\lambda}$.
\item
\[
|W_{\so^-_{2n}}(T_{\mu,\lambda})^F|=C_{\mu,\mu'}\cdot C_{\lambda,\lambda'}\cdot|\BB{S}_{\mu',\lambda'}|\cdot|\BB{S}_{\mu'',\lambda''}|\cdot |W_{\mu,\lambda}|.
\]
\end{itemize}
\end{lemma}
\section{Lusztig correspondence}\label{sec4}
We review some standard facts of Deligne-Lusztig characters and Lusztig correspondence (cf. \cite[Chapter 7, 12]{C}).

\subsection{Deligne-Lusztig characters}
Let $G$ be a connected reductive algebraic
group over $\mathbb{F}_q$. In \cite{DL}, P. Deligne and G. Lusztig defined a virtual character $R^{G}_{T,\theta}$ of $G^F$, associated to an $F$-stable maximal torus $T$ of $G$ and a character $\theta$ of $T^F$.

More generally, let $L$ be an $F$-stable Levi subgroup of a parabolic subgroup $P$ which is not necessarily $F$-stable, and $\pi$ be a representation of the group $L^F$. Then the Deligne-Lusztig induction $R^G_L(\pi)$ is a virtual character of $G^F$. If $P$ is $F$-stable, then the Deligne-Lusztig induction coincides with the parabolic induction
\[
R^G_L(\pi)= \rm{Ind}^{G^F}_{P^F}(\pi).
\]
For example if $T$ is contained in an $F$-stable Borel subgroup $B$, then
\[
R^G_{T,\theta}=\rm{Ind}^{G^F}_{B^F}\theta.
\]

The following facts are standard.

\begin{proposition}[Induction in stages] \label {3.1}
 Let $Q \subset P$ be two parabolic subgroups of $G$, with $F$-stable Levi subgroups $M\subset L$ respectively. Then
\[
R^G_L \circ R_M^L  = R_M^G.
\]
\end{proposition}

\subsection{Lusztig correspondence}

Let $G$ be a connected reductive algebraic
group over $\mathbb{F}_q$.
Let $G^*$ be the dual group of $G$. We still denote the Frobenius endomorphism of $G^*$ by $F$. Then there is a natural bijection between the set of $G^F$-conjugacy classes of $(T, \theta)$ and the set of $G^{*F}$-conjugacy classes of $(T^*, s)$ where $T^*$ is an $F$-stable maximal torus in $G^*$ and $s \in   T^{*F}$. We will also denote $R_{T,\theta}^G$  by $R_{T^*,s}^G$ and write $(T^*,s)=(T,\theta)$ if $(T, \theta)$ corresponds to $(T^*, s)$.

 Assume that $(T,\theta)=(T^*,s)$  and $s$ has eigenvalues $\{x_1,\cdots,x_n\}$. We say $[a]\notin s$ or $[a]\notin \theta$ (resp. $s\sim[a]$ or $\theta\sim[a]$) if $[a]\nsubseteq\{x_1,\cdots,x_n\}$ (resp. $\{x_i\}\subseteq[a]$). We abbreviate $[\pm1]\notin s$ and $[\pm1]\notin \theta$ to $\pm1\notin s$ and $\pm1\notin\theta$, respectively.

For a semisimple element $s \in G^{*F}$, define
\[
\mathcal{E}(G^F,s) = \{ \chi \in \rm{Irr}(G^F)  :  \langle \chi, R_{T^*,s}^G\rangle \ne 0\textrm{ for some }T^*\textrm{ containing }s \}.
\]
The set $\mathcal{E}(G^F,s)$ is called  the Lusztig series. We can thus define a partition of $\rm{Irr}(G^F)$ by Lusztig series
i.e.,
\[
\rm{Irr}(G^F)=\coprod_{(s)}\mathcal{E}(G^F,s).
\]

\begin{proposition}[Lusztig]\label{Lus}
There is a bijection
\[
\mathcal{L}_s:\mathcal{E}(G^F,s)\to \mathcal{E}(C_{G^{*F}}(s),1),
\]
extended by linearity to a map between virtual characters satisfying that
\begin{equation}\label{l}
\mathcal{L}_s(\epsilon_G R^G_{T^*,s})=\epsilon_{C_{G^{*}}(s)} R^{C_{G^{*F}}(s)}_{T^*,1}.
\end{equation}
Moreover, we have
\[
\rm{dim}(\pi)=\frac{|G|_{p'}}{|C_{G^*}(s)|_{p'}}\rm{dim}(\cal{L}_s(\pi))
\]
where $|G|_{p'}$ denotes greatest factor of $|G|$ not divided by $p$, and  $\epsilon_G:= (-1)^r$ where $r$ is the $\Fq$-rank of $G$.
In particular, Lusztig correspondence send cuspidal representation to cuspidal representation.
\end{proposition}

Note that the correspondence $\cal{L}_s$ is usually not uniquely determined.
However, we only consider the uniform case in this paper, and in this case, $\cal{L}_s$ is uniquely determined by (\ref{l}).

\begin{proposition}[Proposition 2.6 in \cite{LW3}]\label{irr}
Let $s$ be a semisimple element of $\UU_n\fq$, which is  $\UU_n\fq$-conjugate to $\rm{diag}(s_1,s_2)$ for some semisimple elements $s_1$ and $s_2$ in $\UU_{n_1}\fq$ and $\UU_{n_2}\fq$ respectively, with $n=n_1+n_2$. Assume that $s_1$ and $s_2$ have no common eigenvalues.  Then for any  $\pi_1\in \mathcal{E}(\UU_{n_1}\fq,s_1)$ and $\pi_2\in\mathcal{E}(\UU_{n_2}\fq,s_2)$, $R^{\UU_n}_{\UU_{n_1}\times\UU_{n_2}}(\pi_1\otimes\pi_2)$ is (up to sign) an irreducible representation. Moreover
\[
R^{\UU_n}_{\UU_{n_1}\times\UU_{n_2}}(\pi_1 \otimes\pi_2) \cong R^{\UU_n}_{\UU_{n_1}\times\UU_{n_2}}(\pi_1'\otimes\pi_2')
\]
if and only if $\pi_1\cong \pi_1'$ and $\pi_2\cong \pi_2'$.
\end{proposition}

Recall that we have a natural decomposition of $C_{G^{*}}(s)$ with the eigenvalues of $s$
\[
C_{G^{*}}(s)=\prod_{[a]}G^*_{[a]}(s).
\]
Hence we obtain a natural decomposition of $T$ as follows: \[T^*=\prod_{[a]}T^*_{[a]}\] where
$T^{*}_{[a]}$ is an $F$-stable maximal torus of $G^*_{[a]}(s)$. Clearly, $T^{*}_{[a]}$ runs over the $F$-stable maximal torus of $G^*_{[a]}(s)$ when $T^*$ runs over the $F$-stable maximal torus of $G^*$ such that $s\in T^*$. Moreover, if two $F$-stable maximal torus $ T^*$ and $T^{\prime *}$ of $G$ with $s\in T^{* F}$ and $s\in T^{\prime * F}$ are in the same $G^F$-conjugacy class, then $\prod_{[a]}T^{*}_{[a]}$ and $\prod_{[a]}T^{\prime *}_{[a]}$ are in the same $\prod_{[a]}G^*_{[a]}(s)^F$-conjugacy class. In fact, assume that ${}^{g^*}T^*=T^{\prime *}$ with $g^*\in G^{*F}$. Then there is a $w\in W_{G^*}(T^{\prime *})^F$ such that ${}^{w}({}^{g^*}s)=s$.
So $wg^*\in C_{G^{*F}}(s)=\prod_{[a]}G^*_{[a]}(s)^F$ and ${}^{wg^*}(\prod_{[a]}T^{*}_{[a]})=\prod_{[a]}T^{\prime *}_{[a]}$.
Applying the above decomposition to the Lusztig correspondence, we have
\[
\mathcal{L}_s(\epsilon_G R^G_{T^*,s})=\prod_{[a]}\epsilon_{G^*_{[a]}(s)} R^{G^*_{[a]}(s)}_{T^*_{[a]},1}.
\]

\begin{proposition}

Let $\pi$ be an irreducible uniform representation in the Lusztig series $\cal{E}(G^F,s)$, and $\pi[a]$ be the restriction of $\mathcal{L}_s(\pi)$ to $ {G^*_{[a]}(s)^F}$.
 Suppose that
 \[
 \pi=\sum_{(T^*,s)\ \textrm{mod }G^F} C_{T^*}  R^G_{T^*,s}
  \]
  and
  \[
  \pi[a]=\sum_{(T^*_{[a]},1)\ \textrm{mod }G^*_{[a]}(s)^F}  C_{T^*_{[a]}}R^{G^*_{[a]}(s)}_{T^*_{[a]},1}
  \]
  where the sum runs over the geometric conjugacy class of $(T,s)$ and $(T^*_{[a]},1)$, respectively, and
$C_T$ and $C_{T^*_{[a]}}$ are the coefficients. Then we have
\begin{equation}\label{coff}
\epsilon_GC_{T^*}=\prod_{[a]} \epsilon_{G^*_{[a]}(s)} C_{T^*_{[a]}}.
\end{equation}
\end{proposition}
\begin{proof}
By Proposition \ref{Lus}, we have
\[
\begin{aligned}
&\mathcal{L}_s(\pi)=\mathcal{L}_s\left(\sum_{(T^*,s)\ \textrm{mod }G^F} C_{T^*}R^G_{T^*,s}\right)=\epsilon_G\sum_{(T^*,s)\ \textrm{mod }G^F} \epsilon_{C_{G^{*}}(s)} C_{T^*} R^G_{T^*,1}\\
=&\epsilon_G\sum_{(T^*,s)\ \textrm{mod }G^F}C_{T^*} \prod_{[a]} \epsilon_{G^*_{[a]}(s)} R^{G^*_{[a]}(s)}_{T^*_{[a]},1}.
\end{aligned}
\]
On the other hand, we know that
\[
  \mathcal{L}_s(\pi)=\prod_{[a]}\pi[a]=\prod_{[a]}\sum_{(T^*_{[a]},1)\ \textrm{mod }G^*_{[a]}(s)^F}  C_{T^*_{[a]}}R^{G^*_{[a]}(s)}_{T^*_{[a]},1}=\sum_{(T^*,s)\ \textrm{mod }G^F} \prod_{[a]}  C_{T^*_{[a]}}R^{G^*_{[a]}(s)}_{T^*_{[a]},1},
  \]
which implies (\ref{coff}).
\end{proof}

\section{Reeder's formula}\label{sec5}
This section aims to recall Reeder's multiplicity formula \cite{R} and explain the notations appearing in this formula.
Let $G$ be a classical group over a finite field $\Fq$, and $H$ be a reductive $\Fq$-subgroup of $G$. In this section, we consider the following three cases:
\begin{itemize}

\item $G=\GGL_{n+1}$ and $H=\GGL_n$.

 \item $G=\UU_{n+1}$ and $H=\UU_n$.

\item $G=\so_{2n+1}$ and $H=\so^\epsilon_{2n}$.
 \end{itemize}
 Let $T$ and $S$ be $F$-stable maximal torus of $G$ and $H$, respectively.

\subsection{Construction of $j_{G_s}$}

Let $H^1(F, W)$ denote the set of $F$-conjugacy classes in $W$. Let $[w] \in H^1(F, W)$ denote the $F$-conjugacy class of an element $w \in W$.
For each $F$-stable maximal torus $T$ of $G$, we obtain a class $\mathrm{cl}(T, G)\in H^1(F, W)$ associate to $T$.
It is easy to check that $\mathrm{cl}(T, G)$ is well-defined. If two $F$-stable maximal tori $T_1$ and $T_2$ are $G^F$-conjugate, then $\mathrm{cl}(T_1, G)=\mathrm{cl}(T_2, G)$.

Let $s \in G^F$ be semisimple, and $G_s := C_G(s)^\circ$ be the identity component of the centralizer $C_G(s)$ of $s$ in $G$. Let $T_s$ be an $F$-stable maximal torus of $G_s$ contained in an $F$-stable Borel subgroup of $G_s$, and $W_{G_s}$ be the Weyl group of $T_s$ in $G_s$.
Then we have a natural map for the set of $G_s^F$-conjugacy classes of $F$-stable maximal tori of $G_s$ to the set of $G^F$-conjugacy classes of $F$-stable maximal tori of $G$ by sending a torus in $G_s$ to the same torus in $G$. This induces a map
 \[j_{G_s} : H^1(F, W_{G_s}) \to H^1(F, W),\]
 and it follows that
\begin{equation}
\mathrm{cl}(T, G)=j_{G_s}(\mathrm{cl}(T, G_s)).\label{2.1}
\end{equation}

 \subsection{A partition of $S$}
 We first recall the notations in \cite{R}. Let $I(S)$ be an index set for the following set of subgroups of $G$,
 \[
 \{G_s:s\in S\}.
 \]
Note that if $G$ is a classical group, then $I(S)$ is finite. For $\iota \in I(S)$, denote by $G_{\iota}$ be the corresponding connected centralizer, and put
 \[
 S_\iota:=\{s\in S:G_s=G_\iota\}.
 \]
 The $F$-action on $S$ induces a permutation of $I(S)$, and let $I(S)^F$ be the $F$-fixed points in $I(S)$. Note that if $S_\iota^F$ is nonempty, then $\iota \in  I(S)^F$.

For $\iota \in  I(S)$,  set
 \[
 H_\iota:=(H\cap G_\iota)^\circ.
 \]
We observe that if $G_s=G_\iota$, then
 \[
 s\in H_\iota\subset G_\iota.
 \]
Put $Z_\iota:=Z(G_\iota)\cap S$. Then it is easy to check that
 \[
 \begin{aligned}
 &Z_\iota\subset Z(H_\iota),\\
 &S_\iota\subset Z_\iota \subset S,\\
 &G_\iota=C_{G}(Z_\iota)^\circ.
 \end{aligned}
 \]

For a semisimple element $s\in G^F$,  put \[N_G(s,T)=\{\gamma\in G :s^{\gamma}\in T\}.\] The group $G^F_s$ acts on $N_G(s, T)^F$ by left multiplication, and we set
\[
 \bar{N}_G(s, T)^F := G^F_s   \backslash N_G(s, T)^F.
 \]
By \cite[Corollary 2.3]{R}, we have an explicit formula for $|\bar{N}_G(s, T)^F |$:
 \begin{lemma}\label{3.0}
 Let $\omega \in H^1(F, W)$, and $T$ be an $F$-stable maximal torus corresponding to $\omega$. Then the set $N_G(s, T)^F$ is nonempty if and only if the fiber $j^{-1}_{G_s}(\omega)$ is nonempty, in which case we have
 \[
 |\bar{N}_G(s, T)^F |=\sum_{\omega'\in j^{-1}_{G_s}(\omega)}\frac{|W_G(T)^F|}{|W_{G_s}(T_{\omega'})^F|},
 \]
 where for each $\omega'\in j^{-1}_{G_s}(\omega)$, the torus $T_{\omega'}$ is an $F$-stable maximal torus corresponding to $\omega'$.
\end{lemma}

\subsection{Reeder's formula}

Let $\chi$ and $\eta$ be two character of $T^F$ and $S^F$, respectively. Fix a $\iota$. For $ v\in j^{-1}_{G_\iota}(\mathrm{cl}(T,G)), \quad \varsigma\in j^{-1}_{H_\iota}(\mathrm{cl}(S,H))$, we set
\begin{equation}\label{chi}
\chi_v:=\frac{1}{|W_{G_\iota}(T)^F|}\sum_{x\in W_{G}(T)^F }(^x\chi)|_{Z_\iota}\quad \mathrm{and}\quad \eta_\varsigma:=\frac{1}{|W_{H_\iota}(S)^F|}\sum_{y\in W_{H}(S)^F }(^y\eta)|_{Z_\iota}.
\end{equation}
 \begin{proposition}\label{proposition:A}
With above $G$ and $H$. We have
 \begin{equation}\label{reeder}
\langle R^G_{T,\chi},R^H_{S,\eta}\rangle _{H^F}=\sum_{\mbox{\tiny$\begin{array}{c}\iota\in I(S)^F,
\delta_\iota=0\\j_{G_\iota}^{-1}(\mathrm{cl}(T,G))\ne \emptyset
\end{array}$}}\frac{(-1)^{\mathrm{rk}(G_\iota)+\mathrm{rk}(H_\iota)+\mathrm{rk}(T)+\mathrm{rk}(S)}}{|\bar{N}_{H}(\iota,S)^F|}\langle \chi_v, \eta_\varsigma\rangle _{Z_\iota^F},
\end{equation}
where $v\in j_{G_\iota}^{-1}(\mathrm{cl}(T,G))$ and $\varsigma\in j_{H_\iota}^{-1}(\mathrm{cl}(S,H))$.
\end{proposition}
\begin{proof}
Proposition 4.4 in \cite{LW1} is the unitary case of this proposition. Similarly, we can prove the formula for the general linear groups case. For the orthogonal  groups case, since \[
|j_{G_\iota}^{-1}(\mathrm{cl}(T,G))||j_{H_\iota}^{-1}(\mathrm{cl}(S,H))|\le 1
 \]
 (see \cite[(p.593)]{R}), the formula (9.1) in \cite{R} is equal to (\ref{reeder}).
\end{proof}
Here we ignore the definition of $\delta_\iota$. We point out that the condition $\delta_\iota=0$ is equal to the following conditions (c.f. \cite{LW1, R} for details):
\begin{itemize}
\item Suppose that $G=\GGL_{n+1}$ and $H=\GGL_n$ (resp. $G=\UU_{n+1}$ and $H=\UU_n$). Then $\delta_\iota=0$ if and only if there is a semisimple element $s\in S_\iota$ conjugate to $\rm{diag}(s', I_{m})$, where $s'$ is a regular semisimple element of $\GGL_{n-m}(\bb{F}_q)$ (resp. $\UU_{n-m}(\bb{F}_q)$) such that $1$ is not an eigenvalue of $s'$. In this case
\begin{equation}\label{g1}
G_\iota\cong G_s\cong T'\times \GGL_{m+1}\ \textrm{(resp. $ T'\times\UU_{m+1}$)},
\end{equation}
where $T'$ is the centralizer of $s'$ in $\GGL_{n-m}$ (resp. $\UU_{n-m}$).

\item Suppose that $G=\so_{2n+1}$ and $H=\so^\epsilon_{2n}$. Then $\delta_\iota=0$ if and only if there is a semisimple element $s\in S_\iota$ conjugate to $\rm{diag}(s', I_{m})$, where $s'$ is a regular semisimple element of $\so^\epsilon_{2(n-m)}(\bb{F}_q)$ such that $1$ is not an eigenvalue of $s'$. In this case
\begin{equation}\label{g3}
G_\iota\cong G_s\cong T'\times \so_{2m+1},
\end{equation}
where $T'$ is the centralizer of $s'$ in $\so^\epsilon_{2(n-m)}$.
 \end{itemize}

 In the above two cases, we have followed equation (see (9.5) in \cite{R} and (4.10) in \cite{LW1})
 \[
 |\bar{N}_{H}(\iota,S)^F|=\frac{|\bar{N}_{H}(S)^F|}{|\bar{N}_{H_\iota}(S)^F|}=\frac{|W_H(S)^F|}{|W_{H_\iota}(S)^F|}.
 \]

For two elements $\iota_1, \iota_2\in I(S)^F$, we say that $\iota_1\sim\iota_2$ if $G_{\iota_1}^F$ and $G_{\iota_2}^F$ are in the same $H^F$-conjugacy class. Denote the equivalence class of $\iota$ by $[\iota]$.  Let $[I(S)^F]$ be the set of equivalence classes in $I(S)^F$.
Let us denote a typical summand in (\ref{reeder}) by $X_\iota$. Similar to \cite[p.14]{LW1}, we get $X_{\iota_1}=X_{\iota_2}$ if $\iota_1\sim\iota_2$.
Then (\ref{reeder})  can be rewritten as
\[
\sum_{\mbox{\tiny$\begin{array}{c}\iota\in I(S)^F,
\delta_\iota=0\\j_{G_\iota}^{-1}(\mathrm{cl}(T,G))\ne \emptyset
\end{array}$}}X_\iota=\sum_{\mbox{\tiny$\begin{array}{c}[\iota]\in [I(S)^F],
\delta_\iota=0\\j_{G_\iota}^{-1}(\mathrm{cl}(T,G))\ne \emptyset
\end{array}$}} \#[\iota]\cdot X_\iota
\]
where
\[
 \begin{aligned}
X_{\iota}&=\frac{(-1)^{\mathrm{rk}(G_\iota)+\mathrm{rk}(H_\iota)+\mathrm{rk}(T)+\mathrm{rk}(S)}}{|\bar{N}_{H}(\iota,S)^F|}\cdot\langle \chi_v, \eta_\varsigma\rangle _{Z_\iota^F}\\
&=\frac{(-1)^{\mathrm{rk}(G_\iota)+\mathrm{rk}(H_\iota)+\mathrm{rk}(T)+\mathrm{rk}(S)}}{|W_{G_\iota}(T)^F|\cdot|W_{H}(S)^F|}\cdot
\sum_{\mbox{\tiny$\begin{array}{c}y\in W_{H}(S)^F \\
x\in W_{G}(T)^F\\
\end{array}$}}
\langle{}^x\chi,{}^y\eta\rangle _{Z_\iota^F}.
\end{aligned}
\]

 By (\ref{g1}) and (\ref{g3}), each $G_\iota$ is of the form
\[
G_\iota\cong T'\times G'.
\]\
 Then
\[
Z_\iota=T'\times I\hookrightarrow T'\times G'.
\]
Consider their dual groups:
\[
Z_\iota^*=T^{\prime *}\times I\hookrightarrow T^{\prime *}\times G^{\prime *}
\]
where $I$ is the identity in $G^{\prime *}$.
 Clearly, a character of $Z_\iota^F$ corresponds to a element in $Z_{\iota}^{*F}\cong T^{\prime *F}$.

  We must calculate the sum $\sum_{x,y}
\langle{}^x\chi,{}^y\eta\rangle _{Z_\iota^F}$. Assume that $(T,\chi)=(T,t)$ and $(S,\eta)=(S,s)$ with $t\in T^{*F}$ and $s\in S^{*F}$. We find that
\[
\sum_{\mbox{\tiny$\begin{array}{c}y\in W_{H}(S)^F \\
x\in W_{G}(T)^F\\
\end{array}$}}
\langle{}^x\chi,{}^y\eta\rangle _{Z_\iota^F}
=\#\left\{(x,y)\in W_{G^*}(T^*)^F\times W_{H^*}(S^*)^F :{}^xt|_{Z^{*F}_\iota}={}^ys|_{Z^{*F}_\iota}\right\}.
\]
Note that $T^{\prime *}$ is determined by $\iota$. Let
\[
D(T^*,t,\iota):=\{t'\in T^{\prime *F}\textrm{ such that }t'=^{w'}t|_{Z_\iota^{*F}}\textrm{ with } w'\in W_{G^*}(T^*)^F\}.
 \]
For $t'\in D(T^*,t,\iota)$, we set
\[
M(T,t,t',\iota):=\{w'\in W_{G^*}(T^*)^F : t'=^{w'}t|_{Z_\iota^{*F}} \}
\]
and
\[
M(S,s,t',\iota):=\{w'\in W_{H^*}(S^*)^F : t'=^{w'}s|_{Z_\iota^{*F}}\}.
\]
Then
\[
\#\left\{(x,y)\in W_{G}(T)^F\times W_{H}(S)^F :{}^xt|_{Z^{*F}_\iota}={}^ys|_{Z^{*F}_\iota}\right\}
=\sum_{t'\in D(T^*,t,\iota)}\#M(T,t,t',\iota)\cdot\#M(S,s,t',\iota).
\]
So we may rewrite $X_{\iota}$ as
\begin{equation}\label{xx}
X_{\iota}=
\frac{(-1)^{\mathrm{rk}(G_\iota)+\mathrm{rk}(H_\iota)+\mathrm{rk}(T)+\mathrm{rk}(S)}}{|W_{G_\iota}(T)^F|\cdot|W_{H}(S)^F|}\cdot
\sum_{t'\in D(T^*,t,\iota)}\#M(T,t,t',\iota)\cdot\#M(S,s,t',\iota).
\end{equation}

\section{unitary group}\label{sec6}
Let $G=\UU_{n+1}$ be a unitary group, and $H=\UU_n\subset G$ be a subgroup of $G$. Let $T$ and $S$ be two $F$-stable maximal torus of $G$ and $H$, respectively. Suppose that $\mu$ and $\mu^S$ be partitions corresponding to $T$ and $S$, respectively. Let $\chi$ and $\eta$ be characters of $T^F$ and $S^F$, respectively. In this section, we will calculate the multiplicity $\langle R^{G}_{T,\chi},R_{S,\eta}^{H}\rangle _{H^F}$. Recall that
\begin{equation}\label{ur}
\langle R^{G}_{T,\chi},R_{S,\eta}^{H}\rangle _{H^F}=
\sum_{\mbox{\tiny$\begin{array}{c}\iota\in I(S)^F,
\delta_\iota=0\\j_{G_\iota}^{-1}(\mathrm{cl}(T,G))\ne \emptyset
\end{array}$}}X_\iota=\sum_{\mbox{\tiny$\begin{array}{c}[\iota]\in [I(S)^F],
\delta_\iota=0\\j_{G_\iota}^{-1}(\mathrm{cl}(T,G))\ne \emptyset
\end{array}$}} \#[\iota]\cdot X_\iota.
\end{equation}

In \cite{LW1}, we know that the set
\[
\{[\iota]\in[I(S)^F]:\delta_{\iota}=0\textrm{ and }j_{G_\iota}^{-1}(\mathrm{cl}(T,G))\ne \emptyset\}
 \]
 is parameterized by pairs $(m,\mu^\iota)$, where $m\leq n$ is a nonnegative integer and $\mu^\iota$ is a partition of $n-m$ such that $\mu^\iota\subset\mu^S$ and $\mu^\iota\subset\mu$. In this case
 \begin{equation}\label{6.1}
G_\iota\cong T'\times \UU_{m+1}\textrm{ and }H_\iota\cong T'\times \UU_{m}
\end{equation}
 where $T'$ is an $F$-stable maximal torus of $\UU_{n-m}$ corresponding to $\mu^\iota$. We denote $\iota$ by $\iota_{\mu'}$ if $\iota$ corresponds to $\mu'$. For simplicity notations, we write $G_{\mu'}$, $H_{\mu'}$ $X_{\mu'}$ instead of $G_{\iota_{\mu'}}$, $H_{\iota_{\mu'}}$ and $X_{\iota_{\mu'}}$.
By \cite[Proposition 4.6]{LW1}, we have
\begin{equation}\label{x}
 \begin{aligned}
\langle R^{G}_{T,\chi},R_{S,\eta}^{H}\rangle _{H^F}
=&\sum_{\mbox{\tiny$\begin{array}{c}\mu'\subset\mu\\
\mu'\subset\mu^S\\
\end{array}$}}C_{\mu^S,\mu'}\cdot X_{\mu'}.
\end{aligned}
 \end{equation}

First, we consider a case where $(S,\eta)=(S,1)$ and $(T,\chi)=(T_1\times T_2,\theta\otimes1)$ with $1\notin \theta$. Let $\mu^i$ be the partition corresponding to $T_i$. It follows from  \cite[subsection 4.6]{LW1} that
$
X_{\mu'}=0 \textrm{ if } \mu'\nsubseteq \mu^2.
$
Then
 \begin{equation}\label{x1}
 \begin{aligned}
\langle R^{G}_{T_1\times T_2,\theta\otimes1},R_{S,1}^{H}\rangle _{H^F}
=&\sum_{\mbox{\tiny$\begin{array}{c}\mu'\subset\mu^2\\
\mu'\subset\mu^S\\
\end{array}$}}C_{\mu^S,\mu'}\cdot X_{\mu'}.
\end{aligned}
 \end{equation}

 Now we consider a similar case as follows. Suppose that $\eta\sim[a]$ and $(T,\chi)=(T_1\times T_2,\theta\otimes\chi')$ with $[a]\notin \theta$ and $\chi'\sim[a]$. Applies the same arguments in \cite[section 5]{LW1}, for this case, we can conclude that
 \begin{equation}\label{x2}
 \begin{aligned}
\langle R^{G}_{T_1\times T_2,\theta\otimes\chi'},R_{S,\eta}^{H}\rangle _{H^F}
=&\sum_{\mbox{\tiny$\begin{array}{c}\mu'\subset\mu^2\\
\mu'\subset\mu^S\\
\end{array}$}}C_{\mu^S,\mu'}\cdot X_{\mu'}.
\end{aligned}
 \end{equation}
 In particular, if $[a]=1$, then (\ref{x2}) becomes (\ref{x1}).

\subsection{Calculation of (\ref{x2})}\label{sec6.1}
Let $G^*$ and $H^*$ be the dual groups of $G$ and $H$, respectively. Let $T_1^*\times T_2^*=T^*$ and $S^*$ be the $F$-stable torus of $G^*$ and $H^*$ corresponding to $T$ and $S$, respectively. Suppose that $(T_1\times T_2,\theta\otimes\chi')=(T_1^*\times T_2^*,t)=(T_1^*\times T_2^*,t_0\times t')$ and $(S,\eta)=(S^*,s)$. In this subsection, we assume that $[a]\sim t'$ and $[a]\sim s$.
Consider the natural embedding:
\[
t_0\times t'\in T_1^{*F}\times T_2^{*F}\hookrightarrow\UU_{|\mu^1|}\fq\times\UU_{|\mu^2|}\fq
\]
and
\[
s\in S^{*F}\hookrightarrow\UU_{n}\fq.
\]
By the decomposition in section \ref{sec2.2}, we have
\[
C_{H^{*F}}(s)=H^*_{[a]}(s)^{F}
\]
and
\[
C_{\UU_{|\mu^2|}\fq}(t')=(\UU_{|\mu^2|})_{[a]}(t')^F
\]
where $H^*_{[a]}(s)$ and $(\UU_{|\mu^2|})_{[a]}(t')$ are either two general linear groups or two unitary groups. More precisely, they are unitary groups if $a^{-1}=a^{q^k}$ for some $k\in\bb{Z}$ and general linear groups otherwise.
Set $h:=\#[a]$. By our discussion in subsection \ref{sec2.2}, the groups $C_{\UU_{|\mu^2|}\fq}(t')$ and $C_{H^{*F}}(s)$ are of the forms as follows:
\[
\left\{
\begin{array}{ll}
C_{\UU_{|\mu^2|}\fq}(t')\cong\UU_{\frac{|\mu^2|}{h}}(\bb{F}_{q^h}),\quad C_{H^{*F}}(s)\cong\UU_{\frac{n}{h}}(\bb{F}_{q^h})&\textrm{ if $h$ is odd};\\
C_{\UU_{|\mu^2|}\fq}(t')\cong\GGL_{\frac{|\mu^2|}{h}}(\bb{F}_{q^h}),\quad C_{H^{*F}}(s)\cong\GGL_{\frac{n}{h}}(\bb{F}_{q^h}),&\textrm{ otherwise.}
\end{array}
\right.
\]
Consider two classical groups $H'\subset G'$ defined over $\bb{F}_{q^h}$ where
\[
\left\{
\begin{array}{ll}
G'=\UU_{\frac{n}{h}+1},\quad H'=\UU_{\frac{n}{h}}&\textrm{ if $h$ is odd};\\
G'=\GGL_{\frac{n}{h}+1},\quad H'=\GGL_{\frac{n}{h}}&\textrm{ otherwise.}
\end{array}
\right.
\]
We still denote the Frobenius endomorphism of $H'$ and $G'$ by $F$.
Then $H^{\prime F}\cong C_{H^{*F}}(s)$, and the torus $S^*\subset C_{H^{*}}(s)$ can be regarded as an $F$-stable maximal torus of $H'$ corresponding to the partition $\frac{\mu^S}{h}$. For $G'$, we have the natural embedding
\[
T_2\times I_{\frac{n-|\mu^2|}{h}+1}\hookrightarrow C_{\UU_{|\mu^2|}}(t')\times I_{\frac{n-|\mu^2|}{h}+1}\hookrightarrow G'
\]
where $I_{\frac{n-|\mu^2|}{h}+1}$ is the identity.
Then there exists an $F$-stable maximal torus $T'=T'_1\times T_2'$ of $G'$ such that $T_1'$ corresponds to the partition $(1,1\cdots,1)$ and $T_2'=T_2$ corresponds to the partition $\frac{\mu^2}{h}$.
Our aim in this subsection is to prove the following proposition which establishes the connection between (\ref{x1}) and (\ref{x2}).
\begin{proposition}
Keep the notations and assumptions as above. We have
\begin{equation}\label{prop6.1}
\langle R^{G'}_{T_1'\times T_2',\theta'\otimes1},R^{H'}_{S^*,1}\rangle _{H^{\prime F}}=\epsilon_{T,t}\cdot\langle R^G_{T_1\times T_2,\theta\otimes\chi},R^H_{S,\eta}\rangle _{H^F}
\end{equation}
where $1 \notin \theta'$ and $\epsilon_{T,t}\in\{\pm 1\}$ depends on $T$ and $t$.
\end{proposition}

\begin{proof}
By (\ref{x2}), the RHS of (\ref{prop6.1}) equals
 \[
 \begin{aligned}
\langle R^{G}_{T_1\times T_2,\theta\otimes\chi'},R_{S,\eta}^{H}\rangle _{H^F}
=&\sum_{\mbox{\tiny$\begin{array}{c}\mu'\subset\mu^2\\
\mu'\subset\mu^S\\
\end{array}$}}C_{\mu^S,\mu'}\cdot X_{\mu'}.
\end{aligned}
 \]
By (\ref{xx}), the typical summand $X_{\mu'}$ is of the form:
\[
X_{\mu'}=
\frac{(-1)^{\mathrm{rk}(G_{\mu'})+\mathrm{rk}(H_{\mu'})+\mathrm{rk}(T)+\mathrm{rk}(S)}}{|W_{G_{\mu'}}(T)^F|\cdot|W_{H}(S)^F|}\cdot
\sum_{t''\in D(T^*,t,\iota_{\mu'})}\#M(T,t,t'',\iota_{\mu'})\cdot\#M(S,s,t'',\iota_{\mu'}).
\]
Applying our discussion and notations of torus in the subsection \ref{sec3.1} and subsection \ref{sec3.2}, we have
\[
\left\{
\begin{aligned}
&D(T^*,t,\iota_{\mu'})=D(\mu,\mu',t)\\
&M(T,t,t'',\iota_{\mu'})=M(\mu,\mu',t,t'')\\
&M(S,s,t'',\iota_{\mu'})=M(\mu^S,\mu',s,t'')\\
\end{aligned}\right.
\]
and
\[
X_{\mu'}=
\frac{(-1)^{\mathrm{rk}(G_{\mu'})+\mathrm{rk}(H_{\mu'})+\mathrm{rk}(T)+\mathrm{rk}(S)}}{|W_{G_{\mu'}}(T)^F|\cdot|W_{H}(S)^F|}\cdot
\sum_{t''\in D(\mu,\mu',t)}\#M(\mu,\mu',t,t'')\cdot\#M(\mu^S,\mu',s,t'').
\]
Recall from Lemma \ref{2u} that there is a bijection from the set $P(\mu,\mu',t)$ to the $W_{\UU_{\mu'}}(T_{\mu'})^F$-conjugacy classes in $D(\mu,\mu',t)$. Let $[t'']$ denote the $W_{\UU_{\mu'}}(T_{\mu'})^F$-conjugacy class of $t''\in D(\mu,\mu',t)$, and $f_{[t'']}$ be the corresponding element in $P(\mu,\mu',t)$. Then
 \begin{equation}\label{ttt}
 \#[t'']=\frac{|W_{\UU_{|\mu'|}}(T_{\mu'})^F|}{|W_{\UU_{|\mu'|}}(T_{\mu'},t'')^F|}.
 \end{equation}
Note that the product $\#M(\mu,\mu',t,t'')\cdot\#M(\mu^S,\mu',s,t'')$ does not change when $t''$ varies over a $W_{G'}(T_{\mu'})^F$-conjugacy class.
 We have
\[
\begin{aligned}
X_{\mu'}=&
\frac{(-1)^{\mathrm{rk}(G_{\mu'})+\mathrm{rk}(H_{\mu'})+\mathrm{rk}(T)+\mathrm{rk}(S)}}{|W_{G_{\mu'}}(T)^F|\cdot|W_{H}(S)^F|}\cdot
\sum_{f_{[t'']}\in P(\mu,\mu',t)}\#[t'']\cdot\#M(\mu,\mu',t,t'')\cdot\#M(\mu^S,\mu',s,t'')\\
=&
\frac{(-1)^{\mathrm{rk}(G_{\mu'})+\mathrm{rk}(H_{\mu'})+\mathrm{rk}(T)+\mathrm{rk}(S)}}{|W_{G_{\mu'}}(T)^F|\cdot|W_{H}(S)^F|}\cdot
\sum_{f_{[t'']}\in P(\mu,\mu',t)}\frac{|W_{\UU_{|\mu'|}}(T_{\mu'})^F|\cdot\#M(\mu,\mu',t,t'')\cdot\#M(\mu^S,\mu',s,t'')}{|W_{\UU_{|\mu'|}}(T_{\mu'},t'')^F|}.\\
\end{aligned}
\]
Since $s\sim [a]$, it follows from Lemma \ref{2u} and (\ref{um}) that $M(\mu^S,\mu',s,t'')\ne 0$ if and only if $t''\sim[a]$, which implies that $M(\mu^S,\mu',s,t'')\ne 0$ if and only if
\[
f_{[t'']}([a'])=\left\{
\begin{array}{ll}
\mu',&\textrm{ if }[a']=[a];\\
0,&\textrm{otherwise }.\\
\end{array}
\right.
\]
In other words, the sum $\sum_{f_{[t'']}\in P(\mu,\mu',t)}$ has only one term, and by abuse of notations, we still denote the index of this term by $f_{[t'']}$. By (\ref{uw}) and (\ref{um}),
\begin{equation}\label{xmp}
\begin{aligned}
X_{\mu'}=&\frac{(-1)^{\mathrm{rk}(G_{\mu'})+\mathrm{rk}(H_{\mu'})+\mathrm{rk}(T)+\mathrm{rk}(S)}}{|W_{G_{\mu'}}(T)^F|\cdot|W_{H}(S)^F|}\cdot
\frac{|W_{\UU_{|\mu'|}}(T_{\mu'})^F|\cdot\#M(\mu,\mu',t,t'')\cdot\#M(\mu^S,\mu',s,t'')}{|W_{\UU_{|\mu'|}}(T_{\mu'},t'')^F|}\\
=&\frac{(-1)^{\mathrm{rk}(G_{\mu'})+\mathrm{rk}(H_{\mu'})+\mathrm{rk}(T)+\mathrm{rk}(S)}\cdot|W_{\UU_{|\mu'|}}(T_{\mu'})^F|}{|W_{G_{\mu'}}(T)^F|\cdot|W_{H}(S)^F|\cdot|W_{\UU_{|\mu'|}}(T_{\mu'},t'')^F|}\\
&\times|W_{\UU_{|\mu'|}}(T_{\mu'},t'')^F|\cdot|W_{G_{\mu'}}(T)^F|\cdot C_{\mu^2,\mu'}\cdot|W_{\UU_{|\mu'|}}(T_{\mu'},t'')^F|\cdot|W_{H_{\mu'}}(T)^F|\cdot C_{\mu^S,\mu'}\\
=&(-1)^{\mathrm{rk}(G_{\mu'})+\mathrm{rk}(H_{\mu'})+\mathrm{rk}(T)+\mathrm{rk}(S)} C_{\mu^2,\mu'}\cdot|W_{\UU_{|\mu'|}}(T_{\mu'},t'')^F|.
\end{aligned}
\end{equation}
Hence, the RHS of (\ref{prop6.1}) becomes
 \begin{equation}\label{xx1}
 \begin{aligned}
\langle R^{G}_{T_1\times T_2,\theta\otimes\chi'},R_{S,\eta}^{H}\rangle _{H^F}
=&\sum_{\mbox{\tiny$\begin{array}{c}\mu'\subset\mu^2\\
\mu'\subset\mu^S\\
\end{array}$}}(-1)^{\mathrm{rk}(G_{\mu'})+\mathrm{rk}(H_{\mu'})+\mathrm{rk}(T)+\mathrm{rk}(S)} C_{\mu^S,\mu'}\cdot C_{\mu^2,\mu'}\cdot|W_{\UU_{|\mu'|}}(T_{\mu'},t'')^F|.
\end{aligned}
 \end{equation}

Now we turn to the LHS of (\ref{prop6.1}). If $h$ is odd, then $G'$ and $H'$ are unitary groups. By \cite[Proposition 4.6 and (5.3)]{LW1}, we have
 \begin{equation}\label{xx11}
\langle R^{G'}_{T_1'\times T_2',\theta'\otimes1},R^{H'}_{S^*,1}\rangle _{H^{\prime F}}=\sum_{\mbox{\tiny$\begin{array}{c}\frac{\mu'}{h}\subset\frac{\mu^2}{h}\\
\frac{\mu'}{h}\subset\frac{\mu^S}{h}\\
\end{array}$}}
(-1)^{\mathrm{rk}\left(G'_{\frac{\mu'}{h}}\right)+\mathrm{rk}\left(H'_{\frac{\mu'}{h}}\right)+\mathrm{rk}(T)+\mathrm{rk}(S)}
C_{\frac{\mu^S}{h},\frac{\mu'}{h}}\cdot C_{\frac{\mu^2}{h},\frac{\mu'}{h}}|W_{\UU_{|\frac{\mu'}{h}|}}(T_{\frac{\mu'}{h}})^F|
 \end{equation}

It is easy to check that
\begin{itemize}

\item There is a isomorphism
\[
C_{\UU_{|\mu'|}}(t'')^F\cong \UU_{|\frac{\mu'}{h}|}^{\prime F}
\]
such that
\[
W_{\UU_{|\mu'|}}(T_{\mu'},t'')^F\cong W_{\UU_{|\frac{\mu'}{h}|}}(T_{\frac{\mu'}{h}})^F;
\]

\item  $C_{\mu^S,\mu'}=C_{\frac{\mu^S}{h},\frac{\mu'}{h}}$ and $C_{\mu^2,\mu'}=C_{\frac{\mu^2}{h},\frac{\mu'}{h}}$;

\item There is a bijection
\[
\left\{\frac{\mu'}{h}\textrm{ such that }\frac{\mu'}{h}\subset\frac{\mu^2}{h},\ \frac{\mu'}{h}\subset\frac{\mu^S}{h}\right\} \to \left\{\mu'\textrm{ such that }\mu'\subset\mu^2,\ \mu'\subset\mu^S\right\}
\]
by sending $\frac{\mu'}{h}$ to $\mu'$.
\end{itemize}
To summarize, we get
 \[
\sum_{\mbox{\tiny$\begin{array}{c}\mu'\subset\mu^2\\
\mu'\subset\mu^S\\
\end{array}$}} C_{\mu^S,\mu'}\cdot C_{\mu^2,\mu'}\cdot|W_{\UU_{|\mu'|}}(T_{\mu'},t'')^F|=\sum_{\mbox{\tiny$\begin{array}{c}\frac{\mu'}{h}\subset\frac{\mu^2}{h}\\
\frac{\mu'}{h}\subset\frac{\mu^S}{h}\\
\end{array}$}}
C_{\frac{\mu^S}{h},\frac{\mu'}{h}}\cdot C_{\frac{\mu^2}{h},\frac{\mu'}{h}}|W_{\UU_{|\frac{\mu'}{h}|}}(T_{\frac{\mu'}{h}})^F|.
 \]
If $h$ is even, which is the general linear group case, we can get the same equation by the same argument in the proof of \cite[Proposition 4.6 and (5.3)]{LW1}.

To compare (\ref{xx1}) and (\ref{xx11}), we only need to consider the signs on both sides. It is easy to check that
\[
\mathrm{rk}(T_2)+\mathrm{rk}(S)=\mathrm{rk}(T_2')+\mathrm{rk}(S^*)
\]
by direct calculation. Here we emphasize that the LHS is $\bb{F}_q$-rank while the RHS is $\bb{F}_{q^h}$-rank. Set
\[
\epsilon_{T,t}':=\mathrm{rk}(T_1)+\mathrm{rk}(T_1')
 \]
 where $\mathrm{rk}(T_1)$ is the $\bb{F}_q$-rank of $T_1$ and $\mathrm{rk}(T_1')$ is the $\bb{F}_{q^h}$-rank of $T_1'$.
We know that
 \[
\mathrm{rk}(\UU_{n+1})+\mathrm{rk}(\UU_{n})=\left\{\begin{array}{ll} 0, & \textrm{if }n\textrm{ is even},\\
1, & \textrm{if }n\textrm{ is odd},
\end{array}\right.
\]
which implies that
\begin{equation}\label{rkuu}
(-1)^{\mathrm{rk}(G_{\mu'})+\mathrm{rk}(H_{\mu'})}=\left\{\begin{array}{ll}
-1, &\textrm{if } n-|\mu'| \textrm{ is odd};\\
1, & \textrm{otherwise.}
\end{array}\right.
\end{equation}

(i) Suppose that $h$ is odd. The groups $G'$ and $H'$ are unitary groups.
Then
\[
(-1)^{\mathrm{rk}\left(G'_{\frac{\mu'}{h}}\right)+\mathrm{rk}\left(H'_{\frac{\mu'}{h}}\right)}=\left\{\begin{array}{ll}
-1, &\textrm{if } \frac{n}{h}-|\frac{\mu'}{h}| \textrm{ is odd};\\
1, & \textrm{otherwise.}
\end{array}\right.
\]
Hence
\[
(-1)^{\mathrm{rk}\left(G'_{\frac{\mu'}{h}}\right)+\mathrm{rk}\left(H'_{\frac{\mu'}{h}}\right)}=(-1)^{\frac{n}{h}-|\frac{\mu'}{h}|}=(-1)^{h(n-|\mu'|)}=(-1)^{\mathrm{rk}(G_{\mu'})+\mathrm{rk}(H_{\mu'})}.
\]
which implies that
\[
\langle R^{G'}_{T_1'\times T_2',\theta'\otimes1},R^{H'}_{S^*,1}\rangle _{H^{\prime F}}=\epsilon_{T,t}\cdot\langle R^G_{T_1\times T_2,\theta\otimes\chi},R^H_{S,\eta}\rangle _{H^F}.
\]
with $\epsilon_{T,t}:=\epsilon_{T,t}'$.

(ii) Suppose that $h$ is even. Then $G'$ and $H'$ are general linear groups. In this case,
\[
(-1)^{\mathrm{rk}\left(G'_{\frac{\mu'}{h}}\right)+\mathrm{rk}\left(H'_{\frac{\mu'}{h}}\right)}\equiv-1.
\]
and
\[
(-1)^{\mathrm{rk}(G_{\mu'})+\mathrm{rk}(H_{\mu'})}=(-1)^{n=|\mu'|}=(-1)^n.
\]
Let $\epsilon_{T,t}:=\epsilon_{T,t}'\cdot (-1)^{n+1}$. We get
\[
(-1)^{\mathrm{rk}\left(G'_{\frac{\mu'}{h}}\right)+\mathrm{rk}\left(H'_{\frac{\mu'}{h}}\right)}\equiv\epsilon_{T,t}\cdot(-1)^{\mathrm{rk}(G_{\mu'})+\mathrm{rk}(H_{\mu'})}.
\]
Hence
\[
\langle R^{G'}_{T_1'\times T_2',\theta'\otimes1},R^{H'}_{S^*,1}\rangle _{H^{\prime F}}=\epsilon_{T,t}\cdot\langle R^G_{T_1\times T_2,\theta\otimes\chi},R^H_{S,\eta}\rangle _{H^F}.
\]
It is worth to point out that $h$ depends on $t$, so $\epsilon_{T,t}$ only depends on $T$ and $t$.

\end{proof}

\subsection{Decompositions of $G$ and $H$}\label{sec6.2}
In subsection \ref{sec6.1}, we have calculated the pairing $\langle R^G_{T,\chi},R^H_{S,\eta}\rangle _{H^F}$ when $\chi$ and $\eta$ are relatively simple. In the rest of this section, we aim to calculate the multiplicity in the general case and deduce the general case from the case in subsection \ref{sec6.1}.

In section \ref{sec2.2}, we have the decompositions:
\[
C_{G^*(\overline{\bb{F}}_q)}(t)=\prod_{[ a]}G^*_{[ a]}(t)(\overline{\bb{F}}_q)
\]
and
\[
C_{H^*(\overline{\bb{F}}_q)}(s)=\prod_{[ a]}H^*_{[ a]}(s)(\overline{\bb{F}}_q)
\]
where $G^*_{[ a]}(t)(\overline{\bb{F}}_q)$ and $H^*_{[ a]}(s)(\overline{\bb{F}}_q)$ are general linear groups if $\#[a]$ is even and are unitary groups otherwise. By our discussion in section \ref{sec3}, we have the natural embedding
\[
T^*=\prod_{[a]}T^*_{\mu[t,a]}\hookrightarrow \prod_{[ a]}G^*_{[ a]}(t)
\]
and
\[
S^*=\prod_{[a]}S^*_{\mu^S[s,a]}\hookrightarrow \prod_{[ a]}H^*_{[ a]}(t)
\]
where $\bigcup_{[a]}\mu[t,a]=\mu$ and $\bigcup_{[a]}\mu^S[s,a]=\mu^S$.

We shall associate to an index $[a]$ a pair of algebra groups $(G'[a],H'[a])$ and a pair of torus $(S'[a],T'[a])$ as follows.
We set
\[
H'[a]:=
\left\{
 \begin{array}{ll}
 G^*_{[ a]}(t)&\textrm{ if }\nu_{[a]}(t)\ge \nu_{[a]}(s);\\
  H^*_{[ a]}(s)&\textrm{ otherwise};
\end{array}
\right.
\]
and
\[
S'[a]:=
\left\{
  \begin{array}{ll}
 T^*_{\mu[t,a]}&\textrm{ if }\nu_{[a]}(t)\ge \nu_{[a]}(s);\\
  S^*_{\mu^S[s,a]}&\textrm{ otherwise}.
\end{array}
\right.
\]
Clearly, $H'[a]=\GGL_{n[a]}$ if $\#[a]$ is even, and $H'[a]=\UU_{n[a]}$ if $\#[a]$ is odd, and $S'[a]$ is an $F$-stable maximal torus of $H'[a]$. Take a classical group $G'[a]$ over $\bb{F}_{q^{\#[a]}}$ as follows:
\[
G'[a]:=
\left\{
  \begin{array}{ll}
 \GGL_{n[a]+1}&\textrm{ if $\#[a]$ is even};\\
  \UU_{n[a]+1}&\textrm{ if $\#[a]$ is odd}.
\end{array}
\right.
\]
The rational points of $G'[a]$ and $H'[a]$ are $\GGL_{n[a]+1}(\bb{F}_{q^{\#[a]}})$ and $\GGL_{n[a]}(\bb{F}_{q^{\#[a]}})$ (resp. $\UU_{n[a]+1}(\bb{F}_{q^{\#[a]}})$ and $\UU_{n[a]}(\bb{F}_{q^{\#[a]}})$) if $\#[a]$ is even (resp. odd).
Then there is a embedding as follows:
\[
\left\{
 \begin{array}{ll}
 S^*_{\mu^S[s,a]}\times I \hookrightarrow H^*_{[ a]}(t)\times I \hookrightarrow G'[a]&\textrm{ if }\nu_{[a]}(t)\ge \nu_{[a]}(s);\\
  T^*_{\mu[t,a]}\times I \hookrightarrow  G^*_{[ a]}(s)\times I \hookrightarrow G'[a]&\textrm{ otherwise}.
\end{array}
\right.
\]
Let
\begin{equation}\label{ta}
T'[a]:=\left\{
 \begin{array}{ll}
 S^*_{\mu^S[s,a]}\times T''[a] &\textrm{ if }\nu_{[a]}(t)\ge \nu_{[a]}(s);\\
  T^*_{\mu[t,a]}\times T''[a] &\textrm{ otherwise}.
\end{array}
\right.
\end{equation}
be an $F$-stable maximal torus of $G'[a]$ with $(-1)^{\rm{rk}(T''[a])}=(-1)^{\#[a]+1}$. The choice of $T''[a]$ and its rank do not matter. Here is a convenient choice for our purpose. Assume the torus $T''[a]$ corresponds to the partition $\mu''[a]$. Then the torus $T'[a]$ and $S'[a]$ correspond to the partitions $\frac{\mu^S[s,a]}{\#[a]}\cup\mu''[a]$ and $\frac{\mu[t,a]}{\#[a]}$ if $\nu_{[a]}(t)\ge \nu_{[a]}(s)$, and to $\frac{\mu[t,a]}{\#[a]}\cup\mu''[a]$ and $\frac{\mu^S[s,a]}{\#[a]}$, otherwise.

\subsection{Calculation of $\langle R^G_{T,\chi},R^H_{S,\eta}\rangle _{H^F}$}\label{secu2}

We will prove that the pairing
$\langle R^G_{T,\chi},R^H_{S,\eta}\rangle _{H^F}$ is equal to a product of some pairings
$\langle R^{G'[a]}_{T'[a],\theta[a]\otimes 1},R^{H'[a]}_{S'[a],1}\rangle _{H^{\prime }[a]^F}$. Here the choice of these characters $\theta[a]$ does not matter, and we only need to assume that they are regular characters satisfying $1\notin\theta[a]$.

First, we compute the typical summand $X_{\mu'}$ in (\ref{x}). With the same argument of subsection \ref{sec6.1}, we have
\begin{equation}\label{63}
 \begin{aligned}
X_{\mu'}
=&\frac{(-1)^{\mathrm{rk}(G_{\mu'})+\mathrm{rk}(H_{\mu'})+\mathrm{rk}(T)+\mathrm{rk}(S)}}{|W_{G_{\mu'}}(T)^F|\cdot|W_{H}(S)^F|}\sum_{t''\in D(T^*,t,\iota_{\mu'})}\#M(T,t,t'',\iota_{\mu'})\cdot\#M(S,s,t'',\iota_{\mu'})\\
=&\frac{(-1)^{\mathrm{rk}(G_{\mu'})+\mathrm{rk}(H_{\mu'})+\mathrm{rk}(T)+\mathrm{rk}(S)}}{|W_{G_{\mu'}}(T)^F|\cdot|W_{H}(S)^F|}\sum_{f_{[t'']}\in P(\mu,\mu',t)}\#[t'']\cdot\#M(\mu,\mu',t,t'')\cdot\#M(\mu^S,\mu',s,t'').\\
\end{aligned}
\end{equation}
It follows from (\ref{um}) that the product $\#M(\mu,\mu',t,t'')\cdot\#M(\mu^S,\mu',s,t'')\ne 0$ if and only if $t''\in D(\mu,\mu',t)\bigcap D(\mu^S,\mu',s)$, which implies $f_{[t'']}\in P(\mu,\mu',t)\bigcap P(\mu^S,\mu',s)$ by Lemma \ref{2u}.
By (\ref{um}) and (\ref{6.1}), one has
\begin{equation}\label{63a}
\#M(\mu,\mu',t,t'')=|W_{\UU_{|\mu'|}}(T_{\mu'},t'')^F|\cdot|W_{G_{\mu'}}(T)^F|\cdot\prod_{[a]}C_{\mu[t,a],\mu'[t'',a]}
\end{equation}
and
\begin{equation}\label{63b}
\#M(\mu^S,\mu',s,t'')=|W_{\UU_{|\mu'|}}(T_{\mu'},t'')^F|\cdot|W_{H_{\mu'}}(T)^F|\cdot\prod_{[a]}C_{\mu^S[s,a],\mu'[t'',a]}.
\end{equation}
Inserting (\ref{cw2}), (\ref{uw}), (\ref{ttt}), (\ref{63a}), and (\ref{63b}) into (\ref{63}), we get
\[
\begin{aligned}
X_{\mu'}
=&\sum_{f_{[t'']}\in P(\mu,\mu',t)\bigcap P(\mu^S,\mu',s)}\frac{(-1)^{\mathrm{rk}(G_{\mu'})+\mathrm{rk}(H_{\mu'})+\mathrm{rk}(T)+\mathrm{rk}(S)}\cdot|W_{\UU_{|\mu'|}}(T_{\mu'})^F|}{|W_{G_{\mu'}}(T)^F|\cdot|W_{H}(S)^F|\cdot|W_{\UU_{|\mu'|}}(T_{\mu'},t'')^F|}\\
&\times|W_{\UU_{|\mu'|}}(T_{\mu'},t'')^F|\cdot|W_{G_{\mu'}}(T)^F|\cdot |W_{\UU_{|\mu'|}}(T_{\mu'},t'')^F|\cdot|W_{H_{\mu'}}(T)^F|\\
&\times\prod_{[a]}C_{\mu[t,a],\mu'[t'',a]}\cdot C_{\mu^S[s,a],\mu'[t'',a]}\\
=&\sum_{f_{[t'']}\in P(\mu,\mu',t)\bigcap P(\mu^S,\mu',s)}\frac{(-1)^{\mathrm{rk}(G_{\mu'})+\mathrm{rk}(H_{\mu'})+\mathrm{rk}(T)+\mathrm{rk}(S)} }{C_{\mu^S,\mu'}}\\
&\times|W_{\UU_{|\mu'|}}(T_{\mu'},t'')^F|\cdot\prod_{[a]}\left(C_{\mu[t,a],\mu'[t'',a]}\cdot C_{\mu^S[s,a],\mu'[t'',a]}\right).
\end{aligned}
\]
Hence, by (\ref{x}), we have
\begin{equation}\label{63c}
 \begin{aligned}
\langle R^{G}_{T,\chi},R_{S,\eta}^{H}\rangle _{H^F}
=&\sum_{\mbox{\tiny$\begin{array}{c}\mu'\subset\mu\\
\mu'\subset\mu^S\\
\end{array}$}}(-1)^{\mathrm{rk}(G_{\mu'})+\mathrm{rk}(H_{\mu'})+\mathrm{rk}(T)+\mathrm{rk}(S)} \\
&\times\sum_{f_{[t'']}\in P(\mu,\mu',t)\bigcap P(\mu^S,\mu',s)} |W_{\UU_{|\mu'|}}(T_{\mu'},t'')^F|\cdot\prod_{[a]}\left(C_{\mu[t,a],\mu'[t,a]}\cdot C_{\mu^S[s,a],\mu's,a]}\right).
\end{aligned}
\end{equation}
We aim to write the pairing $\langle R^{G}_{T,\chi},R_{S,\eta}^{H}\rangle _{H^F}$ as a product. So we need to move $\prod_{[a]}$ forward. The next step is to decompose the group $W_{\UU_{|\mu'|}}(T_{\mu'},t'')^F$ and the sign $(-1)^{\mathrm{rk}(G_{\mu'})+\mathrm{rk}(H_{\mu'})+\mathrm{rk}(T)+\mathrm{rk}(S)}$.

Consider the decomposition
\[
C_{\UU_{|\mu'|}}(t'')=\prod_{[a]}(\UU_{|\mu'|})_{[ a]}(t'').
\]
Recall that we have a natural embedding
\[
\begin{matrix}
T_{\mu'}&=&\prod_{[a]}T_{\mu'[t'',a]}&\hookrightarrow&\prod_{[a]}(\UU_{|\mu'|})_{[ a]}(t'')\\
\\
t''&=&\prod_{[a]}t''[a]&\hookrightarrow&\prod_{[a]}(\UU_{|\mu'|})_{[ a]}(t'').
\end{matrix}
\]
Then
\begin{equation}\label{w}
|W_{\UU_{|\mu'|}}(T_{\mu'},t'')^F|=\prod_{[a]}\left|W_{(\UU_{|\mu'|})_{[ a]}(t'')}(T_{\mu'[t'',a]})^F\right|.
\end{equation}

It is easy to check that
\begin{equation}\label{rkk1}
 (-1)^{\mathrm{rk}(T)+\mathrm{rk}(S)}=\prod_a (-1)^{\mathrm{rk}(T_{\mu[t,a]})+\mathrm{rk}(S_{\mu[s,a]})}
\end{equation}
and
\begin{equation}\label{rku}
(-1)^{\mathrm{rk}(G_{\mu'})+\mathrm{rk}(H_{\mu'})}=\left\{\begin{array}{ll}
-1, &\textrm{if } n-|\mu'| \textrm{ is odd};\\
1, & \textrm{otherwise.}
\end{array}\right.
\end{equation}
Since $n=\sum_{[a]}\nu_{[a]}(s)\cdot\#[a]$, we have
\begin{equation}\label{rkk2}
 (-1)^{\mathrm{rk}(G_{\mu'})+\mathrm{rk}(H_{\mu'})}=(-1)^{n-|\mu'|}=\prod_{[a]}(-1)^{\nu_{[a]}(s)\cdot\#[a]-|\mu'[s,a]|}.
\end{equation}
Inserting (\ref{w}), (\ref{rkk1}), and (\ref{rkk2}) into (\ref{63c}), we get
\[
 \begin{aligned}
\langle R^{G}_{T,\chi},R_{S,\eta}^{H}\rangle _{H^F}
=&\sum_{\mbox{\tiny$\begin{array}{c}\mu'\subset\mu\\
\mu'\subset\mu^S\\
\end{array}$}}\sum_{f_{[t'']}\in P(\mu,\mu',t)\bigcap P(\mu^S,\mu',s)}\\
&\prod_{[a]}(-1)^{\mathrm{rk}(T_{\mu[t,a]})+\mathrm{rk}(S_{\mu[s,a]})+\nu_{[a]}(s)\cdot\#[a]-|\mu'[s,a]|} C_{\mu[t,a],\mu'[t'',a]}\cdot C_{\mu^S[s,a],\mu'[t'',a]}\\
&\times\left|W_{(\UU_{|\mu'|})_{[ a]}(t'')}(T_{\mu'[t'',a]})^F\right|.
\end{aligned}
\]

It follows from the definition of $P(\mu,\mu',t)$ and $P(\mu^S,\mu',s)$ that $f_{[t'']}\in P(\mu,\mu',t)\bigcap P(\mu^S,\mu',s)$ if and only if
\begin{itemize}

\item $f_{[t'']}([a])\subset \mu[t,a]$;

\item  $\ f_{[t'']}([a])\subset\mu^S[s,a]$;

\item  $\bigcup_{[a]}f_{[t'']}([a])=\mu'$.
\end{itemize}
Then
\[
\sum_{\mbox{\tiny$\begin{array}{c}\mu'\subset\mu\\
\mu'\subset\mu^S\\
\end{array}$}}\sum_{f_{[t'']}\in P(\mu,\mu',t)\bigcap P(\mu^S,\mu',s)}\prod_{[a]}
=\prod_{[a]}\sum_{\mbox{\tiny$\begin{array}{c}\lambda[a]\subset\mu[t,a]\\
\lambda[a]\subset\mu^S[s,a]\\
\end{array}$}}
\sum_{f_{[t'']}\in P(\mu[t,a],\lambda[a],t[a])\bigcap P(\mu^S[s,a],\lambda[a],s[a])},
\]
and moreover, the sum $\sum_{f_{[t'']}\in P(\mu[t,a],\lambda[a],t[a])\bigcap P(\mu^S[s,a],\lambda[a],s[a])}$ has only one term, which is
\begin{equation}\label{tt}
f_{[t'']}([a'])=\left\{
\begin{array}{ll}
\lambda[a],&\textrm{ if }[a']=[a];\\
0,&\textrm{otherwise }.\\
\end{array}
\right.
\end{equation}
To summarize, we have follow result:
\begin{proposition}\label{7.4}
Keep notations as above. we have
\begin{equation}\label{xnn}
 \begin{aligned}
\langle R^{G}_{T,\chi},R_{S,\eta}^{H}\rangle _{H^F}
=&\prod_{[a]}\sum_{\mbox{\tiny$\begin{array}{c}\lambda[a]\subset\mu[t,a]\\
\lambda[a]\subset\mu^S[s,a]\\
\end{array}$}}(-1)^{\mathrm{rk}(T_{\mu[t,a]})+\mathrm{rk}(S_{\mu[s,a]})+\nu_{[a]}(s)\cdot\#[a]-|\lambda[a]|}\cdot Y_{\lambda[a]}
\end{aligned}
\end{equation}
where
\begin{equation}\label{xn1}
 \begin{aligned}
Y_{\lambda[a]}:=C_{\mu[t,a],\lambda[a]}\cdot C_{\mu^S[s,a],\lambda[a]}
\left|W_{(\UU_{|\lambda[a]|})_{[a]}(t'')}(T_{\lambda[a]})^F\right|
\end{aligned}
\end{equation}
with $t''\in T_{\lambda[a]}^F$ and $[t'']$ as in (\ref{tt}).
\end{proposition}

\subsection{Decomposition of Reeder's formula}
We now consider the pairing $\langle R^{G'[a]}_{T'[a],\theta[a]\otimes1},R^{H'[a]}_{S'[a],1}\rangle _{H^{\prime }[a]^F}$. Suppose that $G'[a]=\GGL_m$ if $\#[a]$ is even and $G'[a]=\UU_m$ if $\#[a]$ is odd. In the same manner of the subsection \ref{sec6.1}, we have
\begin{equation}\label{631}
 \begin{aligned}
&\langle R^{G'[a]}_{T'[a],\theta[a]\otimes1},R^{H'[a]}_{S'[a],1}\rangle _{H^{\prime }[a]^F}\\
=&\sum_{\mbox{\tiny$\begin{array}{c}\frac{\mu'[a]}{\#[a]}\subset\frac{\mu[t,a]}{\#[a]}\\
\frac{\mu'[a]}{\#[a]}\subset\frac{\mu^S[s,a]}{\#[a]}\\
\end{array}$}}
(-1)^{\mathrm{rk}\left(G'[a]_{\frac{\mu'[a]}{\#[a]}}\right)+\mathrm{rk}\left(H'[a]_{\frac{\mu'[a]}{\#[a]}}\right)+\mathrm{rk}(T'[a])+\mathrm{rk}(S'[a])} Z_{\frac{\mu'[a]}{\#[a]}}
\end{aligned}
\end{equation}
where
\[
Z_{\frac{\mu'[a]}{\#[a]}}:=C_{\frac{\mu^S[s,a]}{\#[a]},\frac{\mu'[a]}{\#[a]}}\cdot C_{\frac{\mu[t,a]}{\#[a]},\frac{\mu'[a]}{\#[a]}}\times\left|W_{G'[a]_{\left|\frac{\mu'[a]}{\#[a]}\right|}}\left(T_{\frac{\mu'[a]}{\#[a]}}\right)^F\right|.
\]
 Recall that $S'[a]$ corresponds to $\frac{\mu[t,a]}{\#[a]}$ if $\nu_{[a]}(t)\ge \nu_{[a]}(s)$, and to $\frac{\mu^S[t,a]}{\#[a]}$ otherwise. However, in (\ref{631}), the roles of these two partitions are the same. So the different choices of partition corresponding to $S'[a]$ give the same equation.

\begin{lemma}\label{61}
\begin{equation}
\sum_{\mbox{\tiny$\begin{array}{c}\lambda[a]\subset\mu[t,a]\\
\lambda[a]\subset\mu^S[s,a]\\
\end{array}$}}Y_{\lambda[a]}=\sum_{\mbox{\tiny$\begin{array}{c}\frac{\mu'[a]}{\#[a]}\subset\frac{\mu[t,a]}{\#[a]}\\
\frac{\mu'[a]}{\#[a]}\subset\frac{\mu^S[s,a]}{\#[a]}\\
\end{array}$}}Z_{\frac{\mu'[a]}{\#[a]}}
\end{equation}
\end{lemma}
\begin{proof}
It is easy to check that
\begin{itemize}

\item There is a isomorphism
\[
C_{\UU_{|\lambda[a]|}}(t''[a])^F\cong G'[a]_{\left|\frac{\lambda[a]}{\#[a]}\right|}^F
\]
such that
\[
W_{(\UU_{|\lambda[a]|})_{[a]}(t'')}(T_{\lambda[a]})^F\cong W_{G'[a]_{\left|\frac{\lambda[a]}{\#[a]}\right|}}(T_{\frac{\lambda[a]}{\#[a]}})^F.
\]

\item  $C_{\mu^S[s,a],\mu'[a]}=C_{\frac{\mu^S[s,a]}{\#[a]},\frac{\mu'[a]}{\#[a]}}$ and $C_{\mu[t,a],\mu'[a]}=C_{\frac{\mu[t,a]}{\#[a]},\frac{\mu'[a]}{\#[a]}}$;

\item There is a bijection
\[
\left\{\frac{\mu'[a]}{\#[a]}\textrm{ such that }\frac{\mu'[a]}{\#[a]}\subset\frac{\mu[t,a]}{\#[a]},\ \frac{\mu'[a]}{\#[a]}\subset\frac{\mu^S[s,a]}{\#[a]}\right\} \to \left\{\lambda[a]\textrm{ such that }\lambda[a]\subset\mu[t,a],\ \lambda[a]\subset\mu^S[s,a]\right\}
\]
by sending $\frac{\mu'[a]}{\#[a]}$ to $\lambda[a]$.
\end{itemize}
\end{proof}
We now turn to compare the ranks in (\ref{xnn}) and in (\ref{631}). We set
 \begin{equation}\label{ua}
\epsilon_{[a]}:=(-1)^{\rm{rk}(T''[a])}=(-1)^{\#[a]+1},
\end{equation}
and by (\ref{ta}), one has
\begin{equation}\label{urk1}
(-1)^{\mathrm{rk}(T'[a])+\mathrm{rk}(S'[a])} =\epsilon_{[a]}\cdot(-1)^{\mathrm{rk}(T_{\mu[t,a]})+\mathrm{rk}(S_{\mu[s,a]})}.
\end{equation}
Here we emphasize that the LHS is $\bb{F}_q$-rank while the RHS is $\bb{F}_{q^{\#[a]}}$-rank.

If $H'[a]$ is a general linear group, then $\#[a]$ is even,
and
\[
(-1)^{\mathrm{rk}\left(G'[a]_{\frac{\mu'[a]}{\#[a]}}\right)+\mathrm{rk}\left(H'[a]_{\frac{\mu'[a]}{\#[a]}}\right)}=-1.
\]
Since $\#{a}$ is even, $|\lambda[a]|$ is also even. Hence
\begin{equation}\label{evenu}
(-1)^{\nu_{[a]}(s)\cdot\#[a]-|\lambda[a]|}=1=(-1)^{\mathrm{rk}\left(G'[a]_{\frac{\mu'[a]}{\#[a]}}\right)+\mathrm{rk}\left(H'[a]_{\frac{\mu'[a]}{\#[a]}}\right)+1}.
\end{equation}
If $H'[a]$ is not a general linear group, then $H'[a]$ is of the form as follows:
\[
H'[a]=
\left\{
 \begin{array}{ll}
 \UU_{\nu_{[a]}(t)\cdot\#[a]}&\textrm{ if }\nu_{[a]}(t)\ge \nu_{[a]}(s);\\
  \UU_{\nu_{[a]}(s)\cdot\#[a]}&\textrm{ otherwise},
\end{array}
\right.
\]
and by (\ref{rkuu}), we have
\[
(-1)^{\mathrm{rk}\left(G'[a]_{\frac{\mu'[a]}{\#[a]}}\right)+\mathrm{rk}\left(H'[a]_{\frac{\mu'[a]}{\#[a]}}\right)}=
\left\{
  \begin{array}{ll}
 (-1)^{\nu_{[a]}(t)\cdot\#[a]-|\mu'[a]|}&\textrm{ if }\nu_{[a]}(t)\ge \nu_{[a]}(s);\\
 (-1)^{\nu_{[a]}(s)\cdot\#[a]-|\mu'[a]|}&\textrm{ otherwise}.
\end{array}
\right.
\]
Then
\begin{equation}\label{oddu}
 \begin{aligned}
&(-1)^{\mathrm{rk}\left(G'[a]_{\frac{\mu'[a]}{\#[a]}}\right)+\mathrm{rk}\left(H'[a]_{\frac{\mu'[a]}{\#[a]}}\right)}\\
=&\left\{
  \begin{array}{ll}
 (-1)^{\nu_{[a]}(s)\cdot\#[a]+1-|\mu'[a]|}&\textrm{ if }\nu_{[a]}(t)\ge \nu_{[a]}(s)\textrm{ and }\nu_{[a]}(t)\equiv \nu_{[a]}(s)+1\ (\textrm{mod }2);\\
 (-1)^{\nu_{[a]}(s)\cdot\#[a]-|\mu'[a]|}&\textrm{ otherwise}.
\end{array}
\right.
 \end{aligned}
\end{equation}
Let $A_{t,s}=\#\{[a]:\#[a]\textrm{ is even}\}$ and $B_{t,s}=\#\{[a]:\nu_{[a]}(t)\ge \nu_{[a]}(s)\textrm{ and }\nu_{[a]}(t)\equiv \nu_{[a]}(s)+1\}$, and
\begin{equation}\label{uts}
\epsilon_{t,s}:=(-1)^{A_{t,s}+B_{t,s}}.
\end{equation}
Then by (\ref{evenu}) and (\ref{oddu}), whether $H'[a]$ is a general linear group or not, we have
\begin{equation}\label{62}
\epsilon_{t,s}\cdot(-1)^{\nu_{[a]}(s)\cdot\#[a]-|\lambda[a]|}=(-1)^{\mathrm{rk}\left(G'[a]_{\frac{\mu'[a]}{\#[a]}}\right)+\mathrm{rk}\left(H'[a]_{\frac{\mu'[a]}{\#[a]}}\right)}
\end{equation}

\begin{proposition}\label{nou}
 Keep the notations and assumptions as above. Then
\begin{equation}
\epsilon_{t,s}\langle R^{G}_{T,\chi},R_{S,\eta}^{H}\rangle _{H^F}=\prod_{[a]}\epsilon_{[a]}\langle R^{G'[a]}_{T'[a],\theta[a]\otimes1},R^{H'[a]}_{S'[a],1}\rangle _{H^{\prime }[a]^F}
\end{equation}
where the product runs over $[a]$ such that $[a]\in \chi$ and $[a]\in \eta$, and it is a finite product.
\end{proposition}
\begin{proof}
It follows from Proposition \ref{7.4}, Lemma \ref{61}, (\ref{urk1}), and (\ref{62}).
\end{proof}

\subsection{Remark on the general linear groups case}\label{sec6.3}

Proposition \ref{nou} is also true for the general linear groups. The proof for the general linear groups is similar and will be left to the reader.

Now we prove a lemma that will be used in a future proof. Let $\pi$ and $\sigma$ be two irreducible unipotent representations of $\GGL_n\fq$ and $\GGL_m\fq$. Assume that $n\ge m$. Suppose that
\[
  \pi=\sum_{(T^*,1)\ \textrm{mod }\GGL_{n}\fq}  C_{T^*}R^{\GGL_{n}}_{T,1}
  \]
  and
 \[
  \sigma=\sum_{(T^{\prime *},1)\ \textrm{mod }\GGL_{m}\fq}  C_{T^{\prime *}}R^{\GGL_{m}}_{T',1}
  \]
  where the sums run over the geometric conjugacy classes of pair $(T^*,1)$ and $(T^{\prime *},1)$.

 Let $s_1,s_2$ be two regula semisimple elements of $\GGL_{n+1-m}\fq$ with $1\notin s_i$, and $\tau_i=\epsilon_i R^{\GGL_{n+1-m}}_{T_i^*,s_i}$ with $\epsilon_i\in\{\pm 1\}$ be an irreducible representation. Then we have
 \[
 I^{\GGL_{n+1}}_{P}(\tau_i\otimes \sigma)=\epsilon_i\sum_{(T^{\prime *}\times T_i^*,(1,s_i))\ \textrm{mod }\GGL_{n+1}\fq}  C_{T'}R^{\GGL_{n+1}}_{T^{\prime *}\times T_i^*,(1,s_i)}
 \]
where $P$ is a an $F$-stable maximal parabolic subgroup with levi factor $\GGL_{m}\times\GGL_{n+1-m}$. Thus
\[
\langle\pi,  I^{\GGL_{n+1}}_{P}(\tau_i\otimes \sigma)\rangle_{\GGL_{n}\fq}= \epsilon_i\sum_{(T^*,1)}   \sum_{(T^{\prime *}\times T_i^*,(1,s_i))} C_{T^*}\cdot C_{T^{\prime *}} \langle R^{\GGL_{n}}_{T^*,1},R^{\GGL_{n+1}}_{T^{\prime *}\times T_i^*,(1,s_i)}\rangle
\]
where the sums run as before. Since (\ref{x}) is still true for general linear groups, we can conclude that the pairing $\langle R^{\GGL_{n}}_{T^*,1},R^{\GGL_{n+1}}_{T^{\prime *}\times T_i^*,(1,s_i)}\rangle$ does not depend on $s_i$.

\begin{lemma}\label{gl}
Let $s$ be a regula semisimple elements of $\GGL_{n+1-m}\fq$ with $1\notin s$, and $\tau=\epsilon R^{\GGL_{n+1-m}}_{T^*,s}$  with $\epsilon\in\{\pm 1\}$ be a irreducible representation. Then the pairing $\langle\pi,  I^{\GGL_{n+1}}_{P}(\tau\otimes \sigma)\rangle_{\GGL_{n}\fq}$ does not depend on $s$. Moreover, the multiplicity $m(\pi,\sigma)$ in (\ref{defgl}) is well defined.
\end{lemma}

\section{Special orthogonal group}\label{sec7}
Let $V$ be a $2n+1$ dimensional space over $\Fq$ with a nondegenerate symmetric bilinear form $(,)$. Fix $v\in V$ with $(v,v)\ne 0$, and $U$ be the orthogonal space of $v$ in $V$. We take $G=\so(V)$ and $H=\so(U)$.
Then $G\cong\so_{2n+1}$ and $H\cong\so_{2n}^\epsilon$.

Let $T$ be an $F$-stable maximal torus of $G$, and $S$ be an $F$-stable maximal torus of $H$. Let $\chi$ (resp. $\eta$) be a character of $T^F$ (resp. $S^F$) with $\pm 1\notin \chi$ and $\pm 1\notin \eta$. Let $t\in T^*$ (resp. $s\in S^*$) be a semisimple element such that $(T,\chi)$ (resp. $(S,\eta)$) corresponds to $(T^*,t)$ (resp. $(S^*,s)$).
We now calculate the pairing $\langle R^G_{T,\chi},R^H_{S,\eta}\rangle _{H^F}$. Suppose that $T$ corresponds to the bipartition $(\mu,\lambda)$ and $S$ corresponds to the bipartition $(\mu^S,\lambda^S)$.
Recall that
\[
\langle R^{G}_{T,\chi},R_{S,\eta}^{H}\rangle _{H^F}=
\sum_{\mbox{\tiny$\begin{array}{c}\iota\in I(S)^F,
\delta_\iota=0\\j_{G_\iota}^{-1}(\mathrm{cl}(T,G))\ne \emptyset
\end{array}$}}X_\iota=\sum_{\mbox{\tiny$\begin{array}{c}[\iota]\in [I(S)^F],
\delta_\iota=0\\j_{G_\iota}^{-1}(\mathrm{cl}(T,G))\ne \emptyset
\end{array}$}} \#[\iota]\cdot X_\iota.
\]

In \cite[section 9]{R}, we know that the set
\[
\{[\iota]\in[I(S)^F]:\delta_{\iota}=0\textrm{ and }j_{G_\iota}^{-1}(\mathrm{cl}(T,G))\ne \emptyset\}
 \]
 is parameterized by triples $(m,\mu^\iota,\lambda^\iota)$, where $\mu^\iota$ and $\lambda^\iota$ are two partitions such that $|\mu^\iota|+|\lambda^\iota|=n-m$. For each $\iota$, we have
 \[
G_\iota\cong T'\times \so_{2m+1}\textrm{ and }H_\iota\cong T'\times \so^\epsilon_{2m}
\]
 where $T'$ is an $F$-stable maximal torus of $\so_{2(n-m)+1}$ corresponding to $(\mu^\iota,\lambda^\iota)$. We denote $\iota$ by $\iota_{\mu',\lambda'}$ if $\iota$ corresponds to the the triple $(m, \mu',\lambda')$. For simplicity notations, we write $G_{\mu',\lambda'}$, $H_{\mu',\lambda'}$ $X_{\mu',\lambda'}$ instead of $G_{\iota_{\mu',\lambda'}}$, $H_{\iota_{\mu',\lambda'}}$ and $X_{\iota_{\mu',\lambda'}}$.

 In \cite[p.594]{R}, we know that
\begin{equation}\label{xso}
 \begin{aligned}
\langle R^{G}_{T,\chi},R_{S,\eta}^{H}\rangle _{H^F}=\sum_{\mbox{\tiny$\begin{array}{c}[\iota]\in [I(S)^F],
\delta_\iota=0\\j_{G_\iota}^{-1}(\mathrm{cl}(T,G))\ne \emptyset
\end{array}$}} \#[\iota]\cdot X_\iota
=\sum_{\mbox{\tiny$\begin{array}{c}\mu'\subset\mu,\ \mu'\subset\mu^S\\
\lambda'\subset\lambda,\ \lambda'\subset\lambda^S\\
\end{array}$}}C_{\mu^S,\mu'}\cdot C_{\lambda^S,\lambda'}\cdot X_{\mu',\lambda'}.
\end{aligned}
 \end{equation}

\subsection{Decompositions of $G$ and $H$}\label{sec7.1}

In section \ref{sec2.2}, we have the natural decomposition as follows:
\[
C_{G^*(\overline{\bb{F}}_q)}(t)=\prod_{[ a]}G^*_{[ a]}(t)(\overline{\bb{F}}_q)
\]
and
\[
C_{H^*(\overline{\bb{F}}_q)}(s)=\prod_{[ a]}H^*_{[ a]}(s)(\overline{\bb{F}}_q).
\]
To be a little more precise, $G^*_{[ a]}(t)(\overline{\bb{F}}_q)$ and $H^*_{[ a]}(s)(\overline{\bb{F}}_q)$ are general linear groups if $a^{-1}\notin[a]$ and are unitary groups otherwise. Since we only care about the pairing $\langle R^{G}_{T,\chi},R_{S,\eta}^{H}\rangle _{H^F}$ with $\pm 1\notin \chi$ and $\pm 1\notin \eta$, from now on, we assume that $G^*_{[  1]}(t)=H^*_{[ 1]}(s)=G^*_{[  -1]}(t)=H^*_{[ -1]}(s)=1$.
By our discussion in section \ref{sec3}, we have the decomposition of torus as follows:
\[
T^*=\left(\prod_{[a]}T^*_{\mu[t,a]}\right)\times\left(\prod_{[a]}T^*_{\lambda[t,a]}\right) \hookrightarrow \prod_{[ a]}G^*_{[ a]}(t)
\]
and
\[
S^*=\left(\prod_{[a]}S^*_{\mu^S[s,a]}\right)\times\left(\prod_{[a]}S^*_{\lambda^S[s,a]}\right)\hookrightarrow \prod_{[ a]}H^*_{[ a]}(t)
\]
where $\bigcup_{[a]}\mu[t,a]=\mu$, $\bigcup_{[a]}\lambda[t,a]=\lambda$, $\bigcup_{[a]}\mu^S[s,a]=\mu^S$, and $\bigcup_{[a]}\lambda^S[s,a]=\lambda^S$.

We shall associate to an index $[a]$ the pair of algebra groups $(G'[a],H'[a])$ and the pair of torus $(S'[a],T'[a])$ as follows.
We set
\[
H'[a]:=
\left\{
 \begin{array}{ll}
 G^*_{[ a]}(t)&\textrm{ if }\nu_{[a]}(t)\ge \nu_{[a]}(s);\\
  H^*_{[ a]}(s)&\textrm{ otherwise};
\end{array}
\right.
\]
and
\[
S'[a]:=
\left\{
  \begin{array}{ll}
 T^*_{\mu[t,a]}\times T^*_{\lambda[t,a]}&\textrm{ if }\nu_{[a]}(t)\ge \nu_{[a]}(s);\\
  S^*_{\mu^S[s,a]}\times S^*_{\lambda^S[s,a]}&\textrm{ otherwise}.
\end{array}
\right.
\]
Clearly, $S'[a]$ is an $F$-stable maximal torus of $H'[a]$, and $H'[a]=\GGL_{n[a]}$ if $a^{-1}\notin [a]$, and $H'[a]=\UU_{n[a]}$ otherwise. Take a algebra group $G'[a]$ defined over $\bb{F}_{q^{\#[a]}}$ as follows:
\[
G'[a]:=
\left\{
  \begin{array}{ll}
 \GGL_{n[a]+1}&\textrm{ if $a^{-1}\notin [a]$};\\
  \UU_{n[a]+1}&\textrm{ otherwise}.
\end{array}
\right.
\]
Then
\[
(G'[a])^F:=
\left\{
  \begin{array}{ll}
 \GGL_{n[a]+1}(\bb{F}_{q^{\#[a]}})&\textrm{ if $a^{-1}\notin [a]$};\\
  \UU_{n[a]+1}(\bb{F}_{q^{\#[a]}})&\textrm{ otherwise}
\end{array}
\textrm{ and }
\right.(H'[a])^F:=
\left\{
  \begin{array}{ll}
 \GGL_{n[a]}(\bb{F}_{q^{\#[a]}})&\textrm{ if $a^{-1}\notin [a]$};\\
  \UU_{n[a]}(\bb{F}_{q^{\#[a]}})&\textrm{ otherwise}.
\end{array}
\right.
\]
We have the embedding as follows:
\[
\left\{
 \begin{array}{ll}
 \left(S^*_{\mu^S[s,a]}\times S^*_{\lambda^S[s,a]}\right)\times I \hookrightarrow H^*_{[ a]}(t)\times I\hookrightarrow G'[a]&\textrm{ if }\nu_{[a]}(t)\ge \nu_{[a]}(s);\\
  \left(T^*_{\mu[t,a]}\times T^*_{\lambda[t,a]}\right)\times I \hookrightarrow  G^*_{[ a]}(s)\times I \hookrightarrow G'[a]&\textrm{ otherwise}.
\end{array}
\right.
\]
Let
\begin{equation}\label{taso}
T'[a]:=\left\{
 \begin{array}{ll}
  \left(S^*_{\mu^S[s,a]}\times S^*_{\lambda^S[s,a]}\right)\times T''[a] &\textrm{ if }\nu_{[a]}(t)\ge \nu_{[a]}(s);\\
  \left(T^*_{\mu[t,a]}\times T^*_{\lambda[t,a]}\right)\times T''[a] &\textrm{ otherwise}.
\end{array}
\right.
\end{equation}
be an $F$-stable maximal torus of $G'[a]$ with $(-1)^{\rm{rk}(T''[a])}=(-1)^{\#[a]+1}$. As in subsection \ref{sec6.2}, the choices of $T''[a]$ and its rank do not matter. Here is a convenient choice for our purpose. Note that now $S'[a]$ and $T'[a]$ are two $F$-stable maximal torus of general linear groups or unitary groups, so they correspond to partitions instead of bipartitions.
Suppose that $T''[a]$ corresponds to $\mu''$. We describe these partitions as follows.
If $G'[a]=\GGL_{n[a]+1}$, then
\[
\left\{
 \begin{array}{ll}
 \lambda^S[s,a]=0&\textrm{ if }\nu_{[a]}(t)\ge \nu_{[a]}(s);\\
  \lambda[t,a]=0&\textrm{ otherwise},
\end{array}
\right.
\]
and $T'[a]$ corresponds to the partition
\[
\left\{
 \begin{array}{ll}
 \frac{\mu^S[s,a]}{\#[a]}\bigcup\mu''&\textrm{ if }\nu_{[a]}(t)\ge \nu_{[a]}(s);\\
  \frac{\mu[t,a]}{\#[a]}\bigcup\mu''&\textrm{ otherwise}.
\end{array}
\right.
\]
In this case, $H'[a]=\GGL_{n[a]}$, and
$S'[a]$ corresponds to the partition
\[
\left\{
 \begin{array}{ll}
 \frac{\mu[t,a]}{\#[a]}&\textrm{ if }\nu_{[a]}(t)\ge \nu_{[a]}(s);\\
  \frac{\mu^S[s,a]}{\#[a]}&\textrm{ otherwise}.
\end{array}
\right.
\]
If $G'[a]=\UU_{n[a]+1}$, then $2|\#[a]$, $\frac{\#[a]}{2}\Big|\lambda[t,a]$ and $\frac{\#[a]}{2}\Big|\lambda^S[s,a]$.
In this case, $T'[a]$ corresponds to the partition
\[
\left\{
 \begin{array}{ll}
 \frac{\mu^S[s,a]}{\#[a]}\bigcup\frac{\lambda^S[s,a]}{\frac{1}{2}\#[a]}\bigcup\mu''&\textrm{ if }\nu_{[a]}(t)\ge \nu_{[a]}(s);\\
  \frac{\mu[t,a]}{\#[a]}\bigcup\frac{\lambda[t,a]}{\frac{1}{2}\#[a]}\bigcup\mu''&\textrm{ otherwise}.
\end{array}
\right.
\]
and
$S'[a]$ corresponds to the partition
\[
\left\{
 \begin{array}{ll}
 \frac{\mu[t,a]}{\#[a]}\bigcup\frac{\lambda[t,a]}{\frac{1}{2}\#[a]}&\textrm{ if }\nu_{[a]}(t)\ge \nu_{[a]}(s);\\
 \frac{\mu^S[s,a]}{\#[a]}\bigcup\frac{\lambda^S[s,a]}{\frac{1}{2}\#[a]}&\textrm{ otherwise}.
\end{array}
\right.
\]
It is worth to point out that
\[
\left\{
 \begin{array}{ll}
 \textrm{$b$ is odd} &\textrm{ if $b\in\frac{\lambda[t,a]}{\frac{1}{2}\#[a]}$ or $b\in\frac{\lambda^S[s,a]}{\frac{1}{2}\#[a]}$};\\
   \textrm{$b$ is even} &\textrm{ if $b\in\frac{\mu[t,a]}{\frac{1}{2}\#[a]}$ or $b\in\frac{\mu^S[s,a]}{\frac{1}{2}\#[a]}$}.
\end{array}
\right.
\]

\subsection{Calculation of $\langle R^G_{T,\chi},R^H_{S,\eta}\rangle _{H^F}$}

As in subsection \ref{secu2}, we will prove that
$\langle R^G_{T,\chi},R^H_{S,\eta}\rangle _{H^F}$ is a product of some
$\langle R^{G'[a]}_{T'[a],\theta[a]\otimes 1},R^{H'[a]}_{S'[a],1}\rangle _{H^{\prime }[a]^F}$. The proof is similar to the unitary group case. Here, the choices of these characters $\theta[a]$ do not matter, and we only assume that they are regular characters satisfying $\pm1\notin\theta[a]$.

By (\ref{xx}), (\ref{xso}) and the similar argument in subsection \ref{secu2}, we have
\begin{equation}\label{eq7}
 \begin{aligned}
\langle R^{G}_{T,\chi},R_{S,\eta}^{H}\rangle _{H^F}
=&\sum_{\mbox{\tiny$\begin{array}{c}\mu'\subset\mu,\ \mu'\subset\mu^S\\
\lambda'\subset\lambda,\ \lambda'\subset\lambda^S\\
\end{array}$}}C_{\mu^S,\mu'}\cdot  C_{\lambda^S,\lambda'} \cdot\frac{(-1)^{\mathrm{rk}(G_{\mu',\lambda'})+\mathrm{rk}(H_{\mu',\lambda'})+\mathrm{rk}(T)+\mathrm{rk}(S)}}{|W_{G_{\mu',\lambda'}}(T)^F|\cdot|W_{H}(S)^F|}\\
&\times\sum_{(f,f')_{[t'']}\in P(\mu,\lambda,\mu',\lambda',t)}\#[t'']\cdot\#M(T,t,t'',\iota_{\mu',\lambda'})\cdot\#M(S,s,t'',\iota_{\mu',\lambda'}).\\
\end{aligned}
\end{equation}
It follows from Lemma \ref{2sp}, (\ref{spm}), Lemma \ref{som}, and Lemma \ref{so2} that
$
\#M(T,t,t'',\iota_{\mu',\lambda'})\cdot\#M(S,s,t'',\iota_{\mu',\lambda'})\ne 0
 $
  if and only if \[t''\in D(\mu,\lambda,\mu',\lambda',t)\bigcap D^\epsilon(\mu^S,\lambda^S,\mu',\lambda',s)\] or equivalently
  \[
  (f,f')_{[t'']}\in P(\mu,\lambda,\mu',\lambda',t)\bigcap P(\mu^S,\lambda^S,\mu',\lambda',s).
  \]
Inserting (\ref{spm}), (\ref{sop}) and (\ref{som2}) into (\ref{eq7}), we have
\[
\begin{aligned}
&\frac{(-1)^{\mathrm{rk}(G_{\mu',\lambda'})+\mathrm{rk}(H_{\mu',\lambda'})+\mathrm{rk}(T)+\mathrm{rk}(S)}}{|W_{G_{\mu',\lambda'}}(T)^F|\cdot|W_{H}(S)^F|}
\sum_{(f,f')_{[t'']}\in P(\mu,\lambda,\mu',\lambda',t)}\#[t'']\cdot\#M(T,t,t'',\iota_{\mu',\lambda'})\cdot\#M(S,s,t'',\iota_{\mu',\lambda'})\\
=&\sum_{(f,f')_{[t'']}\in P(\mu,\lambda,\mu',\lambda',t)\bigcap P(\mu^S,\lambda^S,\mu',\lambda',s)}\frac{(-1)^{\mathrm{rk}(G_{\mu',\lambda'})+\mathrm{rk}(H_{\mu',\lambda'})+\mathrm{rk}(T)+\mathrm{rk}(S)}\cdot|W_{\mu,\lambda}|}{|W_{G_{\mu',\lambda'}}(T)^F|\cdot|W_{H}(S)^F|\cdot|W_{\mu,\lambda,{}^gt''}|}\\
&\times|W_{\sp_{2(|\mu'|+|\lambda'|)}}(T_{\mu',\lambda'},t'')^F|\cdot|W_{G_{\mu',\lambda'}}(T)^F|\cdot \prod_{[a]}C_{\mu[t,a],\mu'[t'',a]}\cdot C_{\lambda[t,a],\lambda'[t'',a]}\\
&\times|\bb{S}_{\mu',\lambda'}|\cdot|\bb{S}_{\mu'',\lambda''}|\cdot |W_{\mu,\lambda,{}^gt''}|\cdot\prod_{[a]}C_{\mu^S[s,a],\mu'[t'',a]}\cdot C_{\lambda^S[s,a],\lambda'[t'',a]}\\
=&\sum_{(f,f')_{[t'']}\in P(\mu,\lambda,\mu',\lambda',t)\bigcap P(\mu^S,\lambda^S,\mu',\lambda',s)}\frac{(-1)^{\mathrm{rk}(G_{\mu',\lambda'})+\mathrm{rk}(H_{\mu',\lambda'})+\mathrm{rk}(T)+\mathrm{rk}(S)} }{C_{\mu^S,\mu'}\cdot C_{\lambda^S,\lambda'}}\\
&\times|W_{\sp_{2(|\mu'|+|\lambda'|)}}(T_{\mu',\lambda'},t'')^F|\cdot\prod_{[a]}\left(C_{\mu[t,a],\mu'[t'',a]}\cdot C_{\lambda[t,a],\lambda'[t'',a]}\cdot C_{\mu^S[s,a],\mu'[t'',a]} C_{\lambda^S[t,a],\lambda'[t'',a]} \right).
\end{aligned}
\]
Hence
\[
 \begin{aligned}
\langle R^{G}_{T,\chi},R_{S,\eta}^{H}\rangle _{H^F}
=&\sum_{\mbox{\tiny$\begin{array}{c}\mu'\subset\mu, \mu'\subset\mu^S\\
\lambda'\subset\lambda, \lambda'\subset\lambda^S\\
\end{array}$}}(-1)^{\mathrm{rk}(G_{\mu',\lambda'})+\mathrm{rk}(H_{\mu',\lambda'})+\mathrm{rk}(T)+\mathrm{rk}(S)} \\
&\times\sum_{(f,f')_{[t'']}\in P(\mu,\lambda,\mu',\lambda',t)\bigcap P(\mu^S,\lambda^S,\mu',\lambda',s)} |W_{\sp_{2(|\mu'|+|\lambda'|)}}(T_{\mu',\lambda'},t'')^F|\\
&\times\prod_{[a]}\left(C_{\mu[t,a],\mu'[t'',a]}\cdot C_{\lambda[t,a],\lambda'[t'',a]}\cdot C_{\mu^S[s,a],\mu'[t'',a]} C_{\lambda^S[t,a],\lambda'[t'',a]} \right).
\end{aligned}
\]

As in subsection \ref{secu2}, we need to move $\prod_{[a]}$ forward and decompose the group $W_{\sp_{2(|\mu'|+|\lambda'|)}}(T_{\mu',\lambda'},t'')^F$ and the sign $(-1)^{\mathrm{rk}(G_{\mu',\lambda'})+\mathrm{rk}(H_{\mu',\lambda'})+\mathrm{rk}(T)+\mathrm{rk}(S)}$.
Consider the decompositions
\[
C_{\sp_{2(|\mu'|+|\lambda'|)}}(t'')=\prod_{[a]}(\sp_{2(|\mu'|+|\lambda'|)})_{[ a]}(t'')
\]
and
\[
T_{\mu',\lambda'}=\prod_{[a]}T_{\mu'[t'',a],\lambda'[t'',a]}\hookrightarrow\prod_{[a]}(\sp_{2(|\mu'|+|\lambda'|)})_{[ a]}(t'').
\]
Here $\left(\sp_{2(|\mu'|+|\lambda'|)}\right)_{[a]}$ is either a general linear group or a unitary group.
Let $t''[a]$ be the part of $t''$ in $T_{\mu'[t'',a],\lambda'[t'',a]}$. Then
\begin{equation}\label{spw}
|W_{\sp_{2(|\mu'|+|\lambda'|)}}(T_{\mu',\lambda'},t'')^F|=\prod_{[a]}\left|W_{(\sp_{2(|\mu'|+|\lambda'|)})_{[ a]}(t'')}(T_{\mu'[t'',a],\lambda'[t'',a]})^F\right|.
\end{equation}

It is easy to check that
\begin{equation}\label{rk0}
 (-1)^{\mathrm{rk}(T)+\mathrm{rk}(S)}=\prod_a (-1)^{\mathrm{rk}(T_{\mu[t,a],\lambda[t,a]})+\mathrm{rk}(S_{\mu^S[s,a],\lambda^S[s,a]})}.
\end{equation}
By \cite[(9.6)]{R}, we have
\begin{equation}\label{rk}
(-1)^{\mathrm{rk}(G_{\mu',\lambda'})+\mathrm{rk}(H_{\mu',\lambda'})}=(-1)^{\rm{rk}(G)+\rm{rk}(H)+|\lambda'|}=(-1)^{\rm{rk}(G)+\rm{rk}(H)}
\prod_{[a]}(-1)^{|\lambda'[t,a]|}.
\end{equation}
Then by (\ref{spw}), (\ref{rk0}), and (\ref{rk}), we get
\[
 \begin{aligned}
\langle R^{G}_{T,\chi},R_{S,\eta}^{H}\rangle _{H^F}
=&(-1)^{\rm{rk}(G)+\rm{rk}(H)}\sum_{\mbox{\tiny$\begin{array}{c}\mu'\subset\mu, \mu'\subset\mu^S\\
\lambda'\subset\lambda, \lambda'\subset\lambda^S\\
\end{array}$}}
\times\sum_{(f,f')_{[t'']}\in P(\mu,\lambda,\mu',\lambda',t)\bigcap P(\mu^S,\lambda^S,\mu',\lambda',s)} \\
&\prod_{[a]}(-1)^{\mathrm{rk}(T_{\mu[t,a],\lambda[t,a]})+\mathrm{rk}(S_{\mu^S[s,a],\lambda^S[s,a]})+|\lambda'[t,a]|} C_{\mu[t,a],\mu'[t'',a]}\cdot C_{\lambda[t,a],\lambda'[t'',a]}\\
&\times  C_{\mu^S[s,a],\mu'[t'',a]}\cdot C_{\lambda^S[t,a],\lambda'[t'',a]}\cdot
\left|W_{(\sp_{2(|\mu'|+|\lambda'|)})_{[ a]}(t'')}(T_{\mu'[t'',a],\lambda'[t'',a]})^F\right|.
\end{aligned}
\]
Note that the sum $\sum_{(f,f')_{[t'']}\in P(\mu,\lambda,\mu',\lambda',t)\bigcap P(\mu^S,\lambda^S,\mu',\lambda',s)} $ has only one term which is
\begin{equation}\label{t}
(f,f')_{[t'']}([a'])=\left\{
\begin{array}{ll}
(\mu'[t'',a],\lambda'[t'',a]),&\textrm{ if }[a']=[a];\\
0,&\textrm{otherwise }.\\
\end{array}
\right.
\end{equation}
So
\[
\sum_{\mbox{\tiny$\begin{array}{c}\mu'\subset\mu, \mu'\subset\mu^S\\
\lambda'\subset\lambda, \lambda'\subset\lambda^S\\
\end{array}$}}
\times\sum_{(f,f')_{[t'']}\in P(\mu,\lambda,\mu',\lambda',t)\bigcap P(\mu^S,\lambda^S,\mu',\lambda',s)}
=\prod_{[a]}\sum_{\mbox{\tiny$\begin{array}{c}\mu'[a]\subset\mu[t,a],\ \mu'[a]\subset\mu^S[s,a]\\
\lambda'[a]\subset\lambda[t,a],\ \lambda'[a]\subset\lambda^S[s,a]\\
\end{array}$}}
\]
To summazize, we have
\begin{proposition}\label{sod}
Keep notations as above, we have
\begin{equation}\label{xn}
 \begin{aligned}
&(-1)^{\rm{rk}(G)+\rm{rk}(H)}\langle R^{G}_{T,\chi},R_{S,\eta}^{H}\rangle _{H^F}\\
=&\prod_{[a]}
\sum_{\mbox{\tiny$\begin{array}{c}\mu'[a]\subset\mu[t,a],\ \mu'[a]\subset\mu^S[s,a]\\
\lambda'[a]\subset\lambda[t,a],\ \lambda'[a]\subset\lambda^S[s,a]\\
\end{array}$}}(-1)^{\mathrm{rk}(T_{\mu[t,a],\lambda[t,a]})+\mathrm{rk}(S_{\mu^S[s,a],\lambda^S[s,a]})+|\lambda'[a]|} \cdot Y_{\mu'[a],\lambda'[a]}
\end{aligned}
\end{equation}
where
\begin{equation}\label{xml}
 \begin{aligned}
Y_{\mu'[a],\lambda'[a]}:=
& C_{\mu[t,a],\mu'[a]}\cdot C_{\lambda[t,a],\lambda'[a]}\cdot C_{\mu^S[s,a],\mu'[a]}\cdot C_{\lambda^S[t,a],\lambda'[a]}\\
&\times
\left|W_{(\sp_{2(|\mu'|+|\lambda'|)})_{[ a]}(t'')}(T_{\mu'[a],\lambda'[a]})^F\right|
\end{aligned}
\end{equation}
with $t''\in T^F_{\mu'[a],\lambda'[a]}$ and $[t'']$ as in (\ref{t}).
\end{proposition}

\subsection{Decomposition of Reeder's formula}
We now consider the pairing $\langle R^{G'[a]}_{T'[a],\theta[a]\otimes1},R^{H'[a]}_{S'[a],1}\rangle _{H^{\prime }[a]^F}$. Suppose that $G'[a]=\UU_m$ if $a=a^{-q^k}$ for some $k$, $G'[a]=\GGL_m$ otherwise. With the same argument of the proof of \cite[Proposition 4.6 and (5.3)]{LW1}, we have
\begin{equation}\label{a}
 \begin{aligned}
&\langle R^{G'[a]}_{T'[a],\theta[a]\otimes1},R^{H'[a]}_{S'[a],1}\rangle _{H^{\prime }[a]^F}\\
 =&\sum_{\mbox{\tiny$\begin{array}{c}\frac{\mu'[a]}{\#[a]}\bigcup \frac{\lambda'[a]}{\frac{1}{2}\#[a]}\subset\frac{\mu[t,a]}{\#[a]}\bigcup \frac{\lambda[a]}{\frac{1}{2}\#[a]}\\
\frac{\mu'[a]}{\#[a]}\bigcup \frac{\lambda'[a]}{\frac{1}{2}\#[a]}\subset\frac{\mu^S[s,a]}{\#[a]}\bigcup \frac{\lambda^S[a]}{\frac{1}{2}\#[a]}\\
\end{array}$}}
(-1)^{\mathrm{rk}\left(G'[a]_{\frac{\mu'[a]}{\#[a]}\bigcup \frac{\lambda'[a]}{\frac{1}{2}\#[a]} }\right)+\mathrm{rk}\left(H'[a]_{\frac{\mu'[a]}{\#[a]}\bigcup \frac{\lambda'[a]}{\frac{1}{2}\#[a]}}\right)+\mathrm{rk}(T'[a])+\mathrm{rk}(S'[a])} \\
&\times Z_{\frac{\mu'[a]}{\#[a]}\bigcup \frac{\lambda'[a]}{\frac{1}{2}\#[a]}}
\end{aligned}
\end{equation}
where
\[
\begin{aligned}
Z_{\frac{\mu'[a]}{\#[a]}\bigcup \frac{\lambda'[a]}{\frac{1}{2}\#[a]}}:=&
C_{\frac{\mu^S[s,a]}{\#[a]}\bigcup \frac{\lambda^S[a]}{\frac{1}{2}\#[a]},\frac{\mu'[a]}{\#[a]}\bigcup \frac{\lambda'[a]}{\frac{1}{2}\#[a]}}\cdot
 C_{\frac{\mu[t,a]}{\#[a]}\bigcup \frac{\lambda[a]}{\frac{1}{2}\#[a]},\frac{\mu'[a]}{\#[a]}\bigcup \frac{\lambda'[a]}{\frac{1}{2}\#[a]}}\\
 &\times\left|W_{G'[a]_{\left|\frac{\mu'[a]}{\#[a]}\bigcup \frac{\lambda^S[a]}{\frac{1}{2}\#[a]}\right|}}\left(T_{\frac{\mu'[a]}{\#[a]}\bigcup \frac{\lambda^S[a]}{\frac{1}{2}\#[a]}}\right)^F\right|.
 \end{aligned}
\]

\begin{lemma}\label{sox}
\[
 \begin{aligned}
 \sum_{\mbox{\tiny$\begin{array}{c}\mu'[a]\subset\mu[t,a],\ \mu'[a]\subset\mu^S[s,a]\\
\lambda'[a]\subset\lambda[t,a],\ \lambda'[a]\subset\lambda^S[s,a]\\
\end{array}$}}Y_{\mu[a],\lambda[a]}
=\sum_{\mbox{\tiny$\begin{array}{c}\frac{\mu'[a]}{\#[a]}\bigcup \frac{\lambda'[a]}{\frac{1}{2}\#[a]}\subset\frac{\mu[t,a]}{\#[a]}\bigcup \frac{\lambda[a]}{\frac{1}{2}\#[a]}\\
\frac{\mu'[a]}{\#[a]}\bigcup \frac{\lambda'[a]}{\frac{1}{2}\#[a]}\subset\frac{\mu^S[s,a]}{\#[a]}\bigcup \frac{\lambda^S[a]}{\frac{1}{2}\#[a]}\\
\end{array}$}}
Z_{\frac{\mu'[a]}{\#[a]}\bigcup \frac{\lambda'[a]}{\frac{1}{2}\#[a]}}
 \end{aligned}
\]
\end{lemma}
\begin{proof}
It is easy to check that
\begin{itemize}

\item There is a isomorphism
\[
(\sp_{2(|\mu'[a]|+|\lambda'[a]|)})_{[ a]}(t'')^F\cong G'[a]_{\left|\frac{\mu'[a]}{\#[a]}\bigcup\frac{\lambda'[a]}{\frac{1}{2}\#[a]}\right|}^F
\]
such that
\[
W_{(\sp_{2(|\mu'[a]|+|\lambda'[a]|)})_{[ a]}(t'')}(T_{\mu'[a],\lambda'[a]})^F\cong W_{G'[a]_{\left|\frac{\mu'[a]}{\#[a]}\bigcup\frac{\lambda'[a]}{\frac{1}{2}\#[a]}\right|}}(T_{\frac{\mu'[a]}{\#[a]}\bigcup\frac{\lambda'[a]}{\frac{1}{2}\#[a]}})^F.
\]

\item We have $C_{\mu^S[s,a],\mu'[a]}=C_{\frac{\mu^S[s,a]}{\#[a]},\frac{\mu'[a]}{\#[a]}}$, $C_{\mu[t,a],\mu'[a]}=C_{\frac{\mu[t,a]}{\#[a]},\frac{\mu'[a]}{\#[a]}}$,
    $C_{\lambda^S[s,a],\lambda'[a]}=C_{\frac{\lambda^S[s,a]}{\frac{1}{2}\#[a]},\frac{\lambda'[a]}{\frac{1}{2}\#[a]}}$ and $C_{\lambda[t,a],\lambda'[a]}=C_{\frac{\lambda[t,a]}{\frac{1}{2}\#[a]},\frac{\lambda'[a]}{\frac{1}{2}\#[a]}}$;
\item We konw that $\frac{\mu[s,a]}{\#[a]}\bigcap \frac{\lambda[a]}{\frac{1}{2}\#[a]}=\emptyset$, and $\frac{\mu^S[s,a]}{\#[a]}\bigcap \frac{\lambda^S[a]}{\frac{1}{2}\#[a]}=\emptyset$. Hence,
    \[
   C_{\frac{\mu^S[s,a]}{\#[a]}\bigcup \frac{\lambda^S[a]}{\frac{1}{2}\#[a]},\frac{\mu'[a]}{\#[a]}\bigcup \frac{\lambda'[a]}{\frac{1}{2}\#[a]}}
   = C_{\frac{\mu^S[s,a]}{\#[a]},\frac{\mu'[a]}{\#[a]}}\cdot  C_{ \frac{\lambda^S[a]}{\frac{1}{2}\#[a]}, \frac{\lambda'[a]}{\frac{1}{2}\#[a]}}
    \]
and
\[
C_{\frac{\mu[t,a]}{\#[a]}\bigcup \frac{\lambda[a]}{\frac{1}{2}\#[a]},\frac{\mu'[a]}{\#[a]}\bigcup \frac{\lambda'[a]}{\frac{1}{2}\#[a]}}=
C_{\frac{\mu[t,a]}{\#[a]},\frac{\mu'[a]}{\#[a]}}\cdot
C_{\frac{\lambda[a]}{\frac{1}{2}\#[a]}, \frac{\lambda'[a]}{\frac{1}{2}\#[a]}}.
\]

\item There is a bijection between
\[
\left\{\frac{\mu'[a]}{\#[a]}\bigcup \frac{\lambda'[a]}{\frac{1}{2}\#[a]}\textrm{ such that }\frac{\mu'[a]}{\#[a]}\bigcup \frac{\lambda'[a]}{\frac{1}{2}\#[a]}\subset\frac{\mu[t,a]}{\#[a]}\bigcup \frac{\lambda[a]}{\frac{1}{2}\#[a]},\ \frac{\mu'[a]}{\#[a]}\bigcup \frac{\lambda'[a]}{\frac{1}{2}\#[a]}\subset\frac{\mu^S[s,a]}{\#[a]}\bigcup \frac{\lambda^S[a]}{\frac{1}{2}\#[a]}\right\}
\]
and
\[
 \left\{\mu'[a]\textrm{ such that }\mu'[a]\subset\mu[t,a],\ \mu'[a]\subset\mu^S[s,a]\right\}\times \left\{\lambda[a]\textrm{ such that }\lambda[a]\subset\lambda[t,a],\ \lambda[a]\subset\lambda^S[s,a]\right\}
\]
by sending $\frac{\mu'[a]}{\#[a]}\bigcup \frac{\lambda'[a]}{\frac{1}{2}\#[a]}$ to $(\mu'[a],\lambda'[a])$.
\end{itemize}
\end{proof}
It remains to conmpare the rank in (\ref{t}) and (\ref{a}). Set
\begin{equation}\label{soa}
\epsilon_{[a]}:=(-1)^{\rm{rk}(T''[a])}=(-1)^{\#[a]+1},
\end{equation} then we get
\[
(-1)^{\mathrm{rk}(T'[a])+\mathrm{rk}(S'[a])} =\epsilon_{[a]}\cdot(-1)^{\mathrm{rk}(T_{\mu[t,a],\lambda[t,a]})+\mathrm{rk}(S_{\mu^S[s,a],\lambda^S[s,a]})}.
\]
Here we emphasize that the LHS is $\bb{F}_q$-rank but the RHS is $\bb{F}_{q^{\#[a]}}$-rank.
If $H'[a]$ is a general linear group, then $\#[a]$ is even, which implies that
\[
(-1)^{\mathrm{rk}\left(G'[a]_{\frac{\mu'[a]}{\#[a]}\bigcup \frac{\lambda'[a]}{\frac{1}{2}\#[a]}}\right)+\mathrm{rk}\left(H'[a]_{\frac{\mu'[a]}{\#[a]}\bigcup \frac{\lambda'[a]}{\frac{1}{2}\#[a]}}\right)}=-1
\]
Since $\#[a]$ is even, $|\lambda'[t'',a]|$ is also even and
\begin{equation}\label{even}
(-1)^{|\lambda'[t'',a]|}=1=(-1)^{\mathrm{rk}\left(G'[a]_{\frac{\mu'[a]}{\#[a]}\bigcup \frac{\lambda'[a]}{\frac{1}{2}\#[a]}}\right)+\mathrm{rk}\left(H'[a]_{\frac{\mu'[a]}{\#[a]}\bigcup \frac{\lambda'[a]}{\frac{1}{2}\#[a]}}\right)+1}.
\end{equation}
If $H'[a]$ is not general linear group, then
\[
H'[a]:=
\left\{
 \begin{array}{ll}
 \UU_{\frac{1}{2}\nu_{[a]}(t)\cdot\#[a]}&\textrm{ if }\nu_{[a]}(t)\ge \nu_{[a]}(s);\\
  \UU_{\frac{1}{2}\nu_{[a]}(s)\cdot\#[a]}&\textrm{ otherwise}.
\end{array}
\right.
\]
and by (\ref{rk})
\begin{equation}\label{odd}
 \begin{aligned}
(-1)^{\mathrm{rk}\left(G'[a]_{\frac{\mu'[a]}{\#[a]}\bigcup \frac{\lambda'[a]}{\frac{1}{2}\#[a]}}\right)+\mathrm{rk}\left(H'[a]_{\frac{\mu'[a]}{\#[a]}\bigcup \frac{\lambda'[a]}{\frac{1}{2}\#[a]}}\right)}&=
\left\{
  \begin{array}{ll}
 (-1)^{\frac{1}{2}\nu_{[a]}(t)\cdot\#[a]}&\textrm{ if }\nu_{[a]}(t)\ge \nu_{[a]}(s);\\
 (-1)^{\frac{1}{2}\nu_{[a]}(s)\cdot\#[a]}&\textrm{ otherwise}.
\end{array}
\right.\\
&=(-1)^{|\lambda'[a]|}.
 \end{aligned}
\end{equation}

Let $A_{t,s}=\#\{[a]:\#[a]\textrm{ is even}\}$ and
\begin{equation}\label{sots}
\epsilon_{t,s}:=(-1)^{A_{t,s}+\rm{rk}(G)+\rm{rk}(H)}.
\end{equation}
By (\ref{even}) and (\ref{odd}), we have
\begin{lemma}\label{sosgn}
\[
\epsilon_{t,s}\cdot(-1)^{|\lambda'[t'',a]|}\cdot (-1)^{\rm{rk}(G)+\rm{rk}(H)}=(-1)^{\mathrm{rk}\left(G'[a]_{\frac{\mu'[a]}{\#[a]}\bigcup \frac{\lambda'[a]}{\frac{1}{2}\#[a]}}\right)+\mathrm{rk}\left(H'[a]_{\frac{\mu'[a]}{\#[a]}\bigcup \frac{\lambda'[a]}{\frac{1}{2}\#[a]}}\right)}.
\]
\end{lemma}

In summarizing, we have the following result.
\begin{proposition}\label{noso}
 Keep the notations and assumptions as above. Then
\begin{equation}
\epsilon_{t,s}\langle R^{G}_{T,\chi},R_{S,\eta}^{H}\rangle _{H^F}=\prod_{[a]}\epsilon_{[a]}\langle R^{G'[a]}_{T'[a],\theta[a]\otimes1},R^{H'[a]}_{S'[a],1}\rangle _{H^{\prime }[a]^F}
\end{equation}
where the product runs over $[a]$ such that $[a]\in \chi$ and $[a]\in \eta$, and it is a finite product.
\end{proposition}
\begin{proof}
It follows immediately from Proposition \ref{sod}, (\ref{a}), Lemma \ref{sox} and Lemma \ref{sosgn}.
\end{proof}

\section{Decomposition of the Gan-Gross-Prasad problem}\label{sec8}

Let $G$ and $H$ be the classical groups in section \ref{sec6} and section
\ref{sec7}. Let $t\in G^F$ and $s\in H^F$ be semisimple elements. Let $\pi\in\cal{E}(G^F,t)$ and $\sigma\in\cal{E}(H^F,s)$. Additional assumptions $\pm1\notin t$ and $\pm1\notin s$ if $G=\so_{2n+1}$. Then $\pi$ and $\sigma$ are uniform, i.e. they are linear combinations of the Deligne-Lusztig characters as follows:
 \[
 \pi=\sum_{(T^*,t)\ \textrm{mod }G^F} C_{\pi,T^*}  R^G_{T^*,t}
  \]
  and
   \[
 \sigma =\sum_{(S^*,s)\ \textrm{mod }H^F} C_{\sigma,S^*}  R^H_{S^*,s}
  \]
 where the sums run over the geometric conjugacy of $(T^*,t)$ and $(S^*,s)$, respectively. We can therefore calculate the pairing $\langle\pi,\sigma\rangle_{H^F}$ by Reeder's formal.

  Consider the irreducible representations $\pi[a]$ and $\sigma[a]$ defined in section \ref{sec4}. They are also linear combinations of the Deligne-Lusztig characters as follows:
    \[
  \pi[a]=\sum_{(T^*_{[a]},1)\ \textrm{mod }G^*_{[a]}(t)^F}  C_{\pi,T^*_{[a]}}R^{G^*_{[a]}(t)}_{T^*_{[a]},1}
  \]
    and
  \[
  \sigma[a]=\sum_{(S^*_{[a]},1)\ \textrm{mod }H^*_{[a]}(s)^F}  C_{\sigma,S^*_{[a]}}R^{H^*_{[a]}(s)}_{S^*_{[a]},1}
  \]
  where the sums run over the geometric conjugacy of $(T^*_{[a]},1)$ and $(S^*_{[a]},1)$, respectively.

By (\ref{coff}), Proposition \ref{nou}, and Proposition \ref{noso}, we have
\[
\begin{aligned}
\langle\pi,\sigma\rangle_{H^F}
=&\sum_{\mbox{\tiny$\begin{array}{c}(T^*,t)\ \textrm{mod }G^F\\
(S^*,s)\ \textrm{mod }H^F\\
\end{array}$}}
C_{\pi,T^*}\cdot C_{\sigma,S^*} \left\langle  R^G_{T^*,t}, R^H_{S^*,s}\right\rangle_{H^F}\\
=&\sum_{\mbox{\tiny$\begin{array}{c}(T^*,t)\ \textrm{mod }G^F\\
(S^*,s)\ \textrm{mod }H^F\\
\end{array}$}}
\epsilon_{t,s}\cdot C_{\pi,T^*}\cdot C_{\sigma,S^*} \prod_{[a]}\epsilon_{[a]}\langle R^{G'[a]}_{T'[a],\theta[a]\otimes1},R^{H'[a]}_{S'[a],1}\rangle _{H^{\prime }[a]^F}\\
=&\epsilon_{t,s}\epsilon_G\epsilon_H\sum_{\mbox{\tiny$\begin{array}{c}(T,t)\ \textrm{mod }G^F\\
(S,s)\ \textrm{mod }H^F\\
\end{array}$}}
 \prod_{[a]}\epsilon_{[a]}\epsilon_{G^*_{[a]}(t)}\epsilon_{H^*_{[a]}(t)} C_{\pi,T^*_{[a]}}\cdot C_{\sigma,S^*_{[a]}} \left\langle R^{G'[a]}_{T'[a],\theta[a]\otimes1},R^{H'[a]}_{S'[a],1}\right\rangle _{H^{\prime }[a]^F}\\
 =&\epsilon_{t,s}\epsilon_G\epsilon_H \prod_{[a]}\epsilon_{[a]}\epsilon_{G^*_{[a]}(t)}\epsilon_{H^*_{[a]}(t)}\sum_{\mbox{\tiny$\begin{array}{c}((T^*_{[a]},1)\ \textrm{mod }G^*_{[a]}(t)^F\\
(S^*_{[a]},1)\ \textrm{mod }H^*_{[a]}(s)^F\\
\end{array}$}}
 C_{\pi,T^*_{[a]}}\cdot C_{\sigma,S^*_{[a]}} \left\langle R^{G'[a]}_{T'[a],\theta[a]\otimes1},R^{H'[a]}_{S'[a],1}\right\rangle _{H^{\prime }[a]^F}.\\
\end{aligned}
\]
As in (\ref{ta}) and $(\ref{taso})$, we write $(T'[a],\theta[a]\otimes1)$ as $(T''[a]\times T^1[a],\theta[a]\otimes1)$. Since $G'[a]=\GGL_{n[a]+1}$ or $\UU_{n[a]+1}$, there is an $F$-stable levi subgroup $G''[a]\times G^1[a]\subset G'[a]$ such that
\[
\begin{matrix}
T''[a]&\times& T^1[a]&&\\
\\
\downarrow&&\downarrow&&\\
\\
G''[a]&\times& G^1[a]&\subset &G'[a],
\end{matrix}
\]
and $T''[a]$ and $T^1[a]$ are $F$-stable maximal torus of $G''[a]$ and $G^1[a]$, respectively.
By the definition of $H'[a]$ and $G^1[a]$, we know that
\[
H'[a]=
\left\{
 \begin{array}{ll}
 G^*_{[ a]}(t)&\textrm{ if }\nu_{[a]}(t)\ge \nu_{[a]}(s);\\
  H^*_{[ a]}(s)&\textrm{ otherwise};
\end{array}
\right.
\]
and
\[
G^1[a]=
\left\{
 \begin{array}{ll}
 H^*_{[ a]}(t)&\textrm{ if }\nu_{[a]}(t)\ge \nu_{[a]}(s);\\
  G^*_{[ a]}(s)&\textrm{ otherwise}.
\end{array}
\right.
\]
Hence
\[
\begin{aligned}
&\left\langle R^{G'[a]}_{T'[a],\theta[a]\otimes1},R^{H'[a]}_{S'[a],1}\right\rangle _{H^{\prime }[a]^F}
  =&\left\langle R^{G'[a]}_{G''[a]\times G^1[a]}\left(R^{G''[a]}_{T''[a],\theta[a]}\otimes R^{G^1[a]}_{T^1[a],1} \right),R^{H'[a]}_{S'[a],1}\right\rangle _{H^{\prime }[a]^F}.\\
\end{aligned}
\]
For our purpose, it is convinced to set $\theta[a]$ to be a regular character for each $[a]$. Assume that $(T'[a],\theta[a]\otimes1)$ corresponds to $(T'[a],x[a])=(T''[a],x''[a])\times(T^1[a],1)$. Then the geometric conjugacy of $(T''[a],x''[a])$ is a single, and there exists an irreducible representation $\tau\in\cal{E}(G''[a]^F,x''[a])$ such that $\tau=\epsilon''_{[a]} R^{G''[a]}_{T''[a],\theta[a]}$ for some $\epsilon''_{[a]}\in\{\pm 1\}$.

Assume that $\nu_{[a]}(t)< \nu_{[a]}(s)$. Then
  \[
  \begin{aligned}
  &\sum_{\mbox{\tiny$\begin{array}{c}((T^*_{[a]},1)\ \textrm{mod }G^*_{[a]}(t)^F\\
(S^*_{[a]},1)\ \textrm{mod }H^*_{[a]}(s)^F\\
\end{array}$}}
 C_{\pi,T^*_{[a]}}\cdot C_{\sigma,S^*_{[a]}}\left\langle R^{G'[a]}_{G''[a]\times G^1[a]}\left(R^{G''[a]}_{T''[a],\theta[a]}\otimes R^{G^1[a]}_{T^1[a],1} \right),R^{H'[a]}_{S'[a],1}\right\rangle _{H^{\prime }[a]^F}\\
=& \sum_{\mbox{\tiny$\begin{array}{c}((T^*_{[a]},1)\ \textrm{mod }G^*_{[a]}(t)^F\\
(S^*_{[a]},1)\ \textrm{mod }H^*_{[a]}(s)^F\\
\end{array}$}}
\left\langle R^{G'[a]}_{G''[a]\times G^1[a]}\left(R^{G''[a]}_{T''[a],\theta[a]}\otimes C_{\pi,T^*_{[a]}} R^{G^1[a]}_{T^1[a],1} \right),C_{\sigma,S^*_{[a]}}R^{H'[a]}_{S'[a],1}\right\rangle _{H^{\prime }[a]^F}\\
 =&\epsilon''_{[a]}
  \left \langle R^{G'[a]}_{G''[a]\times G^1[a]} \left(\tau\otimes\pi[a]\right),\sigma[a]\right\rangle_{H^{\prime }[a]^F}.\\
\end{aligned}
\]
If $G'[a]$ is a general linear group, then by the definition of $m(\pi[a],\sigma[a])$ and Lemma \ref{gl}, we have
\[
\left\langle R^{G'[a]}_{G''[a]\times G^1[a]} \left(\tau\otimes\pi[a]\right),\sigma[a]\right\rangle_{H^{\prime }[a]^F}= m(\pi[a],\sigma[a]).
\]
Suppose that $G'[a]$ is a unitary group. Then $H'[a]$ and $G^1[a]$ are also unitary groups. Assume that $H'[a]=\UU_{n[a]}$ and $G^1[a]=\UU_{m[a]}$. Recall that $\pi[a]$ and $\sigma[a]$ are unipotent representations. If $n[a]-m[a]$ is odd, then by \cite[Theorem 1.2 and Corollary 1.3]{LW3}, we get
\[
\left\langle R^{G'[a]}_{G''[a]\times G^1[a]} \left(\tau\otimes\pi[a]\right),\sigma[a]\right\rangle_{H^{\prime }[a]^F}= m(\pi[a],\sigma[a])
\]
where $m(\pi[a],\sigma[a])$ is the multiplicity in the Bessel case of the Gan-Gross-Prasad problem. Assume that $n[a]-m[a]$ is even.
By \cite[Corollary 5.3 (ii)]{LW3}, for any irreducible generic representation $\tau'$ of $G^{\prime \prime F}[a]$, we have
\[
\left\langle R^{G'[a]}_{G''[a]\times G^1[a]} \left(\tau\otimes\pi[a]\right),\sigma[a]\right\rangle _{H^{\prime }[a]^F}= \pm\left\langle R^{G'[a]}_{G''[a]\times G^1[a]} \left(\tau'\otimes\pi[a]\right),\sigma[a]\right\rangle _{H^{\prime }[a]^F}.
\]
In particular, by Proposition \ref{irr}, we can pick
 \[
\tau'=\epsilon_{\tau'}R^{\UU_{n[a]+1-m[a]}}_{\UU_1\times\UU_{n[a]-m[a]}}(\tau'_1\otimes\tau'_2)
\]
where $\epsilon_{\tau'}\in\{\pm\}$ and $\tau'_1$ and $\tau'_2$ are certain irreducible generic representations of $\UU_1(\bb{F}_{q^{h[a]}})$ and $\UU_{n[a]-m[a]}(\bb{F}_{q^{h[a]}})$ with $h[a]=\frac{\#[a]}{2}\in\bb{Z}$, respectively. Thus
\[
\begin{aligned}
&\left\langle R^{G'[a]}_{G''[a]\times G^1[a]]} \left(\tau\otimes\pi[a]\right),\sigma[a]\right\rangle_{H^{\prime }[a]^F}\\
=&\pm\left \langle R^{\UU_{n[a]+1}}_{\UU_{m[a]+1}\times \UU_{n[a]-m[a]}} \left(\tau'_2\otimes R^{\UU_{m[a]+1}}_{\UU_1\times\UU_{m[a]}}(\tau'_1\otimes\pi[a])\right),\sigma[a]\right\rangle _{H^{\prime }[a]^F}\\
\end{aligned}
\]
Since we already assume that is $q$ large enough such that the main theorem in \cite{S2} holds, we can pick $\tau'_1\in\cal{E}(\UU_1\fq,s')$ such that $1\notin s'$. By Theorem 1.1 and Theorem 1.2 in \cite{LW3}, we have
\[
\begin{aligned}
&\pm\left \langle R^{\UU_{n[a]+1}}_{\UU_{m[a]+1}\times \UU_{n[a]-m[a]}} \left(\tau'_2\otimes R^{\UU_{m[a]+1}}_{\UU_1\times\UU_{m[a]}}(\tau'_1\otimes\pi[a])\right),\sigma[a]\right\rangle _{H^{\prime }[a]^F}\\
=&m\left(R^{\UU_{m[a]+1}}_{\UU_1\times\UU_{m[a]}}(\tau'_1\otimes\pi[a]),\sigma[a]\right)\\
=& m(\pi[a],\sigma[a])
\end{aligned}
\]
where $m(R^{\UU_{m[a]+1}}_{\UU_1\times\UU_{m[a]}}(\tau'_1\otimes\pi[a]),\sigma[a])$ is the multiplicity in the Bessel case, but $m(\pi[a],\sigma[a])$ is in the Fourier-Jacobi case.

For the $\nu_{[a]}(t)\ge \nu_{[a]}(s)$ case, we have a similar result. In summarizing, we have
 \[
  \begin{aligned}
\langle\pi,\sigma\rangle_{H^F} =\pm\prod_{[a]} m(\pi[a],\sigma[a]).
\end{aligned}
\]
Note that both sides are not negative, so we can remove the sign in the RHS, which completes the proof of Theorem \ref{main2}.

\end{document}